\documentclass[12pt]{amsart}

\usepackage{amsfonts}
\usepackage{amsmath, amsthm, amssymb, ulem, amscd}

\def\crn#1#2{{\vcenter{\vbox{
        \hbox{\kern#2pt \vrule width.#2pt height#1pt
           }
          \hrule height.#2pt}}}}
\def\intprod{\mathchoice\crn54\crn54\crn{3.75}3\crn{2.5}2}
\def\into{\mathbin{\intprod}}

\newcommand{\stopthm}{\hfill$\square$\medskip}

\newcommand{\pa}{\partial}
\newcommand{\Ric}{\operatorname{Ric}}
\newcommand{\End}{\operatorname{End}}

\newcommand{\Hol}{\operatorname{Hol}}

\newcommand{\tr}{\operatorname{tr}}
\newcommand{\tf}{\operatorname{tf}}

\newcommand{\R}{\mathbb R}
\newcommand{\N}{\mathbb N}
\newcommand{\C}{\mathbb C}
\newcommand{\J}{\mathbb J}

\newcommand{\fe}{\varphi}
\newcommand{\al}{\alpha}
\newcommand{\la}{\lambda}
\newcommand{\be}{\beta}
\newcommand{\ka}{\kappa}

\newcommand{\Up}{\Upsilon}
\newcommand{\ds}{\diamondsuit}
\newcommand{\gh}{\widehat{g}}

\newcommand{\wh}{\widehat}
\newcommand{\Rt}{\widetilde{R}}
\newcommand{\Jt}{\widetilde{\J}}
\newcommand{\nt}{\widetilde{\nabla}}
\newcommand{\ct}{\widetilde{\chi}}

\newcommand{\Gat}{\widetilde{\Gamma}}
\newcommand{\Gt}{\widetilde{\Gamma}}
\newcommand{\tb}{\bar \theta}

\newcommand{\Ga}{\Gamma}
\newcommand{\gt}{\widetilde{g}}
\newcommand{\cL}{\mathcal{L}}

\newcommand{\cI}{\mathbb I}  

\newcommand{\cG}{\mathcal{G}}
\newcommand{\cGt}{\widetilde{\mathcal{G}}} 
\newcommand{\cM}{\mathcal{M}}
\newcommand{\cB}{\mathcal{B}}

\newcommand{\cU}{\mathcal{U}}

\newcommand{\cT}{\mathcal{T}}
\newcommand{\cD}{\mathcal{D}}
\newcommand{\fg}{\mathfrak{g}}
\newcommand{\fp}{\mathfrak{p}}
\newcommand{\fo}{\mathfrak{o}}
\newcommand{\fs}{\mathfrak{s}}
\newcommand{\nf}{\infty}

\theoremstyle{plain}
\newtheorem{theorem}{Theorem}[section]
\newtheorem{lemma}[theorem]{Lemma}
\newtheorem{proposition}[theorem]{Proposition}

\theoremstyle{definition}

\newtheorem{definition}[theorem]{Definition}

\theoremstyle{remark}

\numberwithin{equation}{section}

\title[Parallel Tractor Extension]{Parallel Tractor Extension and Ambient
  Metrics of Holonomy Split $G_2$}  

\author{C. Robin Graham}
\address{Department of Mathematics, University of Washington,
Box 354350\\
Seattle, WA 98195-4350}
\email{robin@math.washington.edu}

\author{Travis Willse}
\address{Department of Mathematics, University of Washington,
Box 354350\\
Seattle, WA 98195-4350}
\email{willse@math.washington.edu}

\begin{document}

\maketitle

\thispagestyle{empty}

\renewcommand{\thefootnote}{}
\footnotetext{Partially supported by NSF grant \# DMS 0906035.}   
\renewcommand{\thefootnote}{1}

\section{Introduction}\label{intro}
The work of Nurowski and Leistner-Nurowski in \cite{N1}, \cite{N2}, and  
\cite{LN} has revealed a beautiful connection between the geometry of
generic 2-plane fields $\cD$ on manifolds $M$ of dimension 5 and 
pseudo-Riemannian metrics of signature $(3,4)$ in dimension 7 whose
holonomy is the split real form  
$G_2$ of the exceptional Lie group.  (Throughout this paper, unqualified
$G_2$ refers to the split real form.  A 2-plane field on a 5-manifold is
said to be generic if its second commutator spans $TM$ at each point.)  The
starting point is Nurowski's 
observation in \cite{N1} that using Cartan's solution \cite{C} of the 
equivalence problem for such $\cD$, one can invariantly associate
to $\cD$ a conformal class of metrics of signature $(2,3)$ on $M$.  The
ambient metric construction of \cite{FG1} associates to a real-analytic
conformal structure of signature $(p,q)$ on an odd-dimensional manifold $M$ 
a Ricci-flat metric $\gt$ of signature $(p+1,q+1)$ on an open set 
$\cGt\subset \R_+\times M\times \R$ containing $\R_+\times M\times \{0\}$.
Specializing 
to Nurowski's conformal structures produces a metric of signature $(3,4)$
associated to each real-analytic $\cD$.  

In \cite{N2}, Nurowski identifies $\gt$ explicitly for a particular
8-parameter family of generic 2-plane fields on $\R^5$.  The family is  
parametrized by $\R^8$ via polynomial equations.  This is  
remarkable in itself, as it is rare that $\gt$ can be explicitly
identified.   
Using Nurowski's formula for $\gt$, Leistner-Nurowski show in 
\cite{LN} that all the metrics in the family satisfy $\Hol(\gt)\subset
G_2$, and that if one of four of the eight
parameters is nonzero, then $\Hol(\gt)=G_2$. 
In particular, this gives a completely explicit 8-parameter family of
metrics of holonomy $G_2$.  

One reason this result is of interest is because metrics whose holonomy
equals $G_2$ are not easy 
to come by.  Despite their appearance on Berger's list in 1955, it was not
until 1987 that Robert Bryant \cite{Br1} first proved the existence of such
metrics.  
The case of compact $G_2$ has received more attention in the intervening
years, due partly to the role of manifolds of holonomy compact $G_2$ 
as an $M$-theory analogue of Calabi-Yau manifolds.  But even in that case 
where more is known, new constructions are of interest.     

The analysis of Leistner-Nurowski of the holonomy of these metrics depends
crucially on 
Nurowski's explicit formula for $\gt$.  It is natural to ask whether
analogous 
properties hold for more general $\cD$.  In this paper we show this is
the case. 

For simplicity we take $M$ to be oriented (which is equivalent to $\cD$ 
being oriented).  Our results extend easily to the non-orientable case.  
Our first main result extends the Leistner-Nurowski holonomy containment to
general real-analytic $\cD$. 
\begin{theorem}\label{holonomycontained}
Let $\cD\subset TM$ be a generic 2-plane field on a
connected, oriented 5-manifold $M$, with $M$ and $\cD$ real-analytic.  
Then $\Hol(\gt)\subset G_2$.   
\end{theorem}
Our second main result provides sufficient conditions for $\Hol(\gt)= G_2$.     
We impose two pointwise nondegeneracy conditions involving the Weyl and 
Cotton tensors of a representative $g$ of Nurowski's conformal class.  If
$x\in M$, define the linear transformation $L_x:T_xM\times \R \rightarrow 
\otimes^3T_x^*M$ by
\begin{equation}\label{Lform}
L_x(v,\lambda)=W_{ijkl}v^i+C_{jkl}\lambda,
\end{equation}
where $W_{ijkl}$ denotes the Weyl tensor of $g$ at $x$ and $C_{jkl}$ the
Cotton tensor.  
The map $L_x$ depends on the choice of representative $g$ of the conformal 
class, but the conformal transformation laws of $W$ and $C$ show that
its range, and therefore also its rank, are invariant under conformal
rescaling.  Our first condition is that $L_x$ has rank 6, or equivalently,
that it is injective.   

Our second condition depends only on the Weyl tensor; in fact it depends
only on the 5-dimensional piece of the Weyl tensor giving Cartan's 
basic curvature invariant $A$ of generic 2-plane fields whose vanishing
locally characterizes the homogeneous model.  $A$ is a section of
$S^4\cD^*$, i.e. it is a 
symmetric 4-form on $\cD$.  For $y\in M$, we say that $A_y$ is 
3-nondegenerate if the only vector $X\in \cD_y$ satisfying $A(Y,X,X,X)=0$ for
all $Y\in \cD_y$ is $X=0$.  We will say that $A_y$ is 3-degenerate if it is
not 3-nondegenerate. 
\begin{theorem}\label{holonomycriterion}
Let $\cD\subset TM$ be a generic 2-plane field on a
connected, oriented 5-manifold $M$, with $M$ and $\cD$ real-analytic.  
Suppose
there exist $x,y\in M$ such that $L_x$ is injective and $A_y$ is
3-nondegenerate.  Then $\Hol(\gt)= G_2$.   
\end{theorem}  

In Theorems~\ref{holonomycontained} and \ref{holonomycriterion}, the domain
of $\gt$ is taken to be a sufficiently small neighborhood of $\R_+\times
M\times \{0\}$ in $\R_+\times M\times \R$ diffeomorphic to 
$\R_+\times M\times \R$, which is invariant under dilations 
in the $\R_+$ variable.  

As regards Theorem~\ref{holonomycriterion}, we also show that if one fixes 
a point of a 5-manifold $M$, most generic rank 2 
distributions satisfy that $L$ is injective and $A$ is
3-nondegenerate at that point.  There is a normal form
for such distributions:  with respect to a suitable choice of local
coordinates $(x,y,z,p,q)$, an 
arbitrary generic 2-plane field can be written locally as 
\begin{equation}\label{Dform}
\cD=\operatorname{span}\{\pa_q, \pa_x+p\pa_y+q\pa_p+F\pa_z\}
\end{equation}
for a scalar function $F$ such that $F_{qq}$ is nonvanishing (see
\cite{BH}).  The
coordinates can be taken so that the chosen point is the origin.
For $\cD$ in the form \eqref{Dform}, in \cite{N1}, \cite{N2}, Nurowski
gives a formula for a representative metric $g_F$ of the conformal class 
such that the components of $g_F$ and $g_F^{-1}$ are polynomials in $F$,
the derivatives of $F$ of orders $\leq 4$, and $F_{qq}^{-1}$, with 
coefficients which are universal functions of the local coordinates.
{From} this it is evident that the components of the Weyl  
and Cotton tensors of $g_F$ can be expressed by similar formulae 
involving derivatives of $F$ of orders $\leq 6$, resp. $\leq 7$.  
\begin{proposition}\label{Fhol}
Let $\cD$ have the form \eqref{Dform} with $F_{qq}$ nonvanishing.  Each of
the sets defined by the 
conditions $\operatorname{rank}(L)<6$ at the origin, or $A$
is 3-degenerate at the origin, is contained in a proper algebraic
subvariety in the space of 7-jets of $F$ at the origin.  
\end{proposition}

The union of these two sets is therefore contained in a proper
Zariski closed set in the space of 7-jets at the origin, so its complement
is dense.  The Taylor expansion of $F$ beyond order 7 can be
chosen arbitrarily without affecting these conditions.  By
Theorem~\ref{holonomycriterion}, 
if at any point of $M$ the 7-jet of an $F$ representing $\cD$ lies in this
complement, then $\Hol(\gt)=G_2$.  In fact, it is 
sufficient that the conditions be violated at different points.   

The ambient metrics arising from distributions satisfying the
hypotheses in Theorem~\ref{holonomycriterion} thus form an
infinite-dimensional family of 
metrics whose holonomy is equal to $G_2$.  Unfortunately, these metrics  
are not complete.  They are ``global'' with respect to $M$, but
arise as power series in the variable $\rho$ in the last $\R$ factor whose
radius of convergence may be small.  

As explained in Chapter 4 of \cite{FG2}, the restriction of an ambient  
metric to $\{\rho>0\}$ or $\{\rho<0\}$ is a cone metric over a base which
is called a Poincar\'e metric $g_+$.  (This is another reason they are not
complete.)  The $G_2$ holonomy condition on $\gt$ can be   
reintepreted in terms of $g_+$.  Depending on the sign chosen for $\rho$, 
$g_+$ either has signature $(2,4)$ and is nearly K\"ahler of constant type
1, or has signature 
$(3,3)$ and is nearly para-K\"ahler of constant type 1.  In particular,
this gives new infinite-dimensional families of such metrics.  This and 
other consequences of parallel tractor extension in the Poincar\'e
metric setting will be the subject of a forthcoming paper by the second
author.    

The conditions in Theorem~\ref{holonomycriterion} are far from necessary
for $\Hol(\gt)=G_2$.  Our goal was to find simple conditions which could be 
verified at a point and which give a dense set of 2-plane fields for which  
$\Hol(\gt)=G_2$.  Nurowski's whole family of examples lies in our
complementary set even though almost all of them satisfy $\Hol(\gt)=G_2$.     
In fact, any $\cD$ in Nurowski's full 8-parameter family has the property
that 
the tensor $A$ is 2-degenerate at every point:  at each point 
there exists $0\neq X\in \cD$ such that $A(Y_1,Y_2,X,X)=0$ for all $Y_1$,
$Y_2\in \cD$.  

Theorem~\ref{holonomycriterion} is proved by quoting 
Theorem~\ref{holonomycontained} and then using arguments developed by 
Leistner-Nurowski to rule out the 
possibility that the holonomy is strictly contained in $G_2$. In order to
prove Theorem~\ref{holonomycontained}, we establish a result in a much more
general setting which we think is of independent interest.  We explain this  
next.        

The analog in conformal geometry of the Levi-Civita connection in
Riemannian geometry is the tractor connection.  On a conformal 
manifold $(M,c)$  of signature $(p,q)$, $p+q=n$, there is a canonical
rank $n+2$ vector bundle $\cT$, the standard tractor bundle.  It carries a 
metric of signature $(p+1,q+1)$ and a canonical connection, the normal
tractor connection, with respect to which the tractor metric is parallel.
The holonomy of the tractor connection is therefore a subgroup of
$O(p+1,q+1)$ and is 
referred to as the conformal holonomy of $(M,c)$.  Just as interesting 
classes of pseudo-Riemannian metrics can be described by holonomy 
reductions, interesting classes of conformal structures can be described 
by conformal holonomy reductions.  A result of
Hammerl-Sagerschnig \cite{HS} shows that Nurowski's conformal structures
have precisely such a characterization:  an oriented conformal
structure of signature $(2,3)$ arises from a generic 2-plane field $\cD$ if
and only if its conformal holonomy is contained in $G_2$.

Just as pseudo-Riemannian holonomy reductions are often characterized by
the existence of a parallel tensor, conformal holonomy
reductions are often characterized by the existence of a parallel
tractor (by which we here mean a parallel section of $\otimes^r \cT^*$ 
for some $r>0$).  Since  
$G_2$ is defined as the subgroup of $GL(7,\R)$ preserving a 3-form
compatible with a metric of signature $(3,4)$, the Hammerl-Sagerschnig   
holonomy criterion can be reinterpreted 
(as they do) as the condition that $(M,c)$ admit a parallel tractor 3-form
(i.e. a section of $\Lambda^3\cT^*$) compatible with the tractor metric. 
Likewise, in order to show that $\Hol(\gt)\subset G_2$, one needs to show
that there is a parallel 3-form on the ambient space compatible with 
$\gt$.  Thus one is led to the problem of constructing a parallel 
3-form on the ambient space given a parallel tractor 3-form.

There are several constructions of the standard tractor bundle and its
metric and connection.  These were originally defined by T. Y. Thomas in 
\cite{T} in language that predates the definition of a vector 
bundle.  The first modern treatment is \cite{BEG}.  The paper \cite{CG1}
explains the relation between tractors and the ambient construction.  
Recall that the ambient metric
$\gt$ is defined on an open subset $\cGt\subset \R_+\times M\times\R$
containing $\R_+\times M\times \{0\}$.  The
hypersurface $\cG=\R_+\times M\times \{0\}\subset \cGt$ can be viewed as an
$\R_+$-bundle over $M$.  A tractor (section of $\cT$) on $M$ can   
be regarded as a section of the bundle $T\cGt|_{\cG}$ over $\cG$ with a
particular homogeneity 
with respect to the $\R_+$-dilations.  The tractor metric and connection
can be realized as the restriction to $\cG$ of
the ambient metric $\gt$ and its Levi-Civita connection.  That is, the 
restriction to $\cG$ of a tensor
field on $\cGt$ homogeneous of the correct degree with respect to the
$\R_+$ dilations defines a tractor field on $M$.  The condition that the 
tractor field is parallel with respect to the tractor connection is 
precisely the condition that the restriction to $\cG$ of the tensor field
have zero covariant derivative with respect to the ambient connection when
the differentiations are 
taken in directions tangent to $\cG$.  Thus the problem described above of
constructing a parallel 3-form on the ambient space given a parallel
tractor 3-form amounts to extending such a ``tangentially parallel''
ambient 3-form defined on the hypersurface $\cG$ to a parallel 3-form on
$\cGt$.   

We prove a ``parallel tractor extension theorem'' of this nature for
general tractors irrespective of their rank, symmetry or algebraic 
type on conformal manifolds of any dimension and signature.  For smooth
odd dimensional conformal manifolds, the ambient metric 
is determined by the conformal structure to infinite order
along $\cG$.  For smooth even-dimensional conformal manifolds, it is only
determined to order $n/2-1$.  These indeterminacies in the ambient metric
are reflected in the statement of the parallel tractor extension
theorem.  

\begin{theorem}\label{main}
Let $(M,c)$ be a conformal manifold of dimension $n\geq 3$ and let $\gt$ be
an ambient metric for $(M,c)$.  Let $r\in \N$ and    
suppose $\chi\in\Gamma\left(\otimes^r\mathcal{T}^*\right)$ satisfies
$\nabla\chi =0$, where $\nabla$ denotes the tractor covariant derivative.   
\begin{itemize}
\item
If $n$ is odd, then $\chi$ has an ambient extension $\ct$
satisfying $\nt \ct = O(\rho^\infty)$.  
\item
If $n$ is even, then $\chi$ has an ambient extension satisfying 
$\nt \ct = O(\rho^{n/2-1})$.     
\end{itemize}
\end{theorem}

\noindent
The proof shows that for $n$ odd, if $(M,c)$ and $\gt$ are real  
analytic, then $\ct$ 
may be taken to be real-analytic so that $\nt\ct=0$ in a neighborhood
of $\cG$.   

Theorem~\ref{main} was proved for $n$ odd in \cite{Go3} for the case $r=1$
using an argument based on ``harmonic extension''.  The same argument
also proves Theorem~\ref{main} for $n$ even and $r=1$.  Results
essentially containing Theorem~\ref{main} in case $c$ contains an Einstein
metric are proved in \cite{Leit1}, \cite{Leis}; see the discussion in
\S\ref{neven} below.  

Our proof of Theorem~\ref{main} goes by first extending $\chi$ by parallel 
translation along the lines $\rho\mapsto (z,\rho)$ for each $z\in 
\cG$. 
It must be shown that this parallel translation preserves the 
vanishing covariant derivatives in the tangential directions.  This
involves commutation arguments which use consequences to high order of
the homogeneity and Ricci-flatness of the ambient metric.  

In \cite{Br2}, Bryant showed that for generic 3-plane fields on 6-manifolds 
there is a construction analogous to Nurowski's construction:  a  
generic 3-plane field induces a conformal structure of signature $(3,3)$ on
the same 6-manifold.    
It is tempting to speculate about the possibility of constructing signature
$(4,4)$ metrics of holonomy $\operatorname{Spin}(3,4)$ as ambient metrics
of such conformal manifolds.  This would require finding extensions 
parallel to infinite order for $n$ even in Theorem~\ref{main}.  

Partly with such considerations in mind, we investigate here
some beginning cases of what can happen in even dimensions concerning  
parallel extension beyond order $n/2-1$.  There are   
several issues.  One complication is that in order to construct ambient  
metrics which are Ricci-flat to higher order, it is in general necessary to 
include log terms in the expansion of $\gt$.  To avoid this complication we
mostly restrict attention here to the case of vanishing obstruction tensor,
for which log terms do not enter.  In this case, the proof of  
Theorem~\ref{main} shows that $n/2$ is the critical order:  if a parallel
tractor has an ambient extension satisfying $\nt \ct = O(\rho^{n/2})$, 
it has an extension satisfying $\nt \ct = O(\rho^\infty)$.  Another
complication in even dimensions is that higher-order 
ambient metrics are no longer determined by the conformal structure alone:
there is an ambiguity at order $n/2$ in the ambient metric.  So whether or
not a parallel tractor 
has a parallel ambient extension may depend on which ambient metric one
chooses.  In \S\ref{neven}  
we investigate this for three classes of conformal structures 
admitting parallel tractors:  conformal classes containing an
Einstein metric, Poincar\'e-Einstein conformal classes, and Fefferman
conformal structures associated to nondegenerate hypersurfaces in $\C^n$.   
We find that for conformal classes containing an Einstein metric, there is
always a unique choice of ambiguity for which there is a 
parallel ambient extension (Proposition~\ref{canonicalparallel}).  For 
Poincar\'e-Einstein conformal classes we give necessary and 
sufficient conditions on the Poincar\'e-Einstein metric for there to exist 
an ambient metric for which the parallel 
tractor has a parallel ambient extension (Proposition~\ref{evenpe}), 
and using a recent result of  
Juhl (\cite{J}) we give a formula for the distinguished ambient metric and
the parallel extension when they exist 
(Proposition~\ref{peextension}).  In particular,  Proposition~\ref{evenpe}
gives examples of parallel tractors which have no ambient extension parallel
to order $n/2$ for any choice of ambiguity.  For Fefferman conformal
structures we   
show that the parallel tractor 2-form has a parallel extension for
infinitely many choices of ambiguity in the ambient metric
(Proposition~\ref{crextension}).    

The organization of the paper is as follows.  In \S\ref{ambtrac} we review
background material concerning ambient metrics and tractors.  In
\S\ref{parext} we prove Theorem~\ref{main}.  As a consequence of 
Theorem~\ref{main} we derive a sequence of integrability conditions to  
higher and higher order which must be satisfied by any parallel
tractor.  
\S\ref{neven} studies parallel extension beyond
order $n/2 -1$ in even dimensions as described above.  We formulate a
general condition on a tractor which we call ``determining'' which  
guarantees in the case of vanishing obstruction tensor that there is at
most one choice of ambient metric with respect 
to which the parallel tractor has an ambient extension parallel to order
$n/2$. 
In \S\ref{G2holonomy} we discuss background concerning generic 2-plane
fields, Nurowski's conformal structures, and the work of Leistner-Nurowski
and Hammerl-Sagerschnig.  We show how Cartan's tensor $A$ can be realized
as a piece of the Weyl curvature of Nurowski's conformal structure and we
prove Theorems~\ref{holonomycontained} and \ref{holonomycriterion} and
Proposition~\ref{Fhol}.  \S\ref{appendix} is an appendix in which we
give our conventions concerning $G_2$, collect facts about Cartan's
connection and curvature in the form given by Nurowski \cite{N1}, and prove
two facts about the curvature which are used in \S\ref{G2holonomy}.   

Some of the results in this paper are contained in the Ph.D. thesis of the
second author (\cite{W}).  This thesis includes more details and further
results in certain directions.

\section{Ambient Metrics and Tractors}\label{ambtrac}
In this section we review background material concerning ambient metrics
and tractors.  The main reference for the material on ambient metrics is
\cite{FG2}, and 
references for the approach taken here for tractors are \cite{CG1} and 
\cite{BG}. 

Let $(M,c)$ be a conformal manifold of dimension $n\geq 3$ and signature
$(p,q)$, $p+q=n$.  This means that $c$ is an equivalence class of metrics
under the 
relation $g\sim \Omega^2g$ for $0<\Omega\in C^\infty(M)$.  The
metric bundle of $(M,c)$ is by definition $\cG:=\{(x,g_x):x\in M, g\in 
c\}\subset S^2T^*M$.  Let $\pi:\cG\rightarrow M$ denote the projection.  
There is an action of $\R_+$ on $\cG$ defined by 
$\delta_s(x,g_x)=(x,s^2g_x)$ for $s\in \R_+$.  Let 
$T=\frac{d}{ds}\delta_s|_{s=1}$ be the infinitesimal generator of the 
$\R_+$ action.  There is a tautological symmetric 2-tensor 
${\bf g}_0$ on $\cG$ defined for $X$, $Y\in T_{(x,g_x)}\cG$ by 
${\bf g}_0(X,Y)=g_x(\pi_*X,\pi_*Y)$. 

Consider the space $\cG\times \R$.  The variable in the
$\R$ factor is usually denoted $\rho$.  The dilations $\delta_s$ extend to
$\cG\times \R$ acting in the $\cG$ factor, and we denote also by $T$ the
infinitesimal generator on $\cG\times\R$.  The map 
$\iota:\cG\rightarrow\cG\times\R$ defined for $z\in \cG$ by
$\iota(z)=(z,0)$ imbeds $\cG$ as a hypersurface in $\cG\times\R$.     

A smooth metric $\gt$ of signature 
$(p+1,q+1)$ on a dilation-invariant open neighborhood $\cGt$ of
$\cG\times\{0\}$ in $\cG\times \R$ is said to be a pre-ambient metric for
$(M,c)$ if it satisfies the following two conditions: 
\begin{enumerate}
\item[(1)] $\delta_s^* \gt =s^2 \gt\quad$ for $s\in \R_+$;
\item[(2)] $\iota^* \gt={\bf g}_0$.
\end{enumerate}
A pre-ambient metric $\gt$ is said to be straight if for each $p\in\cGt$
the parametrized curve 
$s\mapsto \delta_sp$ is a geodesic for $\gt$.  This is equivalent to the
condition that $\nt T =Id$ where $Id$ denotes the identity endomorphism and
$\nt$ the Levi-Civita connection of $\gt$; 
see Propositions 2.4 and 3.4 of \cite{FG2}.    

If $n$ is odd, an ambient metric for $(M,c)$ is a straight pre-ambient
metric for $(M,c)$ such that $\Ric(\gt)$ vanishes to infinite order on
$\cG\times \{ 0 \}$.  (The straightness condition is automatic to infinite
order, but it is convenient to include it in the definition.)  There exists
an ambient metric for $(M,c)$ and it is unique to 
infinite order up to pullback by a diffeomorphism defined on a
dilation-invariant neighborhood of $\cG\times\R$ which commutes with
dilations and which restricts to the identity on $\cG\times\{0\}$.    
If $M$ is a real-analytic manifold and there is a real-analytic metric in
the conformal class,
then there exists a real-analytic 
ambient metric for $(M,c)$ satisfying $\Ric(\gt)=0$ on some
dilation-invariant $\cGt$ as above.

In order to formulate the definition of ambient metrics for $n$ even, 
if $S_{IJ}$ is a symmetric 2-tensor field on an 
open neighborhood of $\cG \times \{ 0 \}$ in $\cG \times  \R$ and 
$m \geq 0$, we write $S_{IJ} = O^+_{IJ}( \rho^m)$ if
$S_{IJ} = O(\rho^m)$ and 
for each point $z\in \cG$, the symmetric 2-tensor $(\iota^*(\rho^{-m}S))(z)$  
is of the form 
$\pi^*s$ for some symmetric 2-tensor $s$ at $x=\pi(z)\in M$ satisfying 
$\operatorname{tr}_{g_x}s = 0$.  The symmetric 2-tensor $s$ is allowed to
depend on $z$, not just on $x$.  If $n$ is even, an ambient metric for
$(M,c)$ is a straight pre-ambient metric 
such that $\Ric(\gt)=O^+_{IJ}( \rho^{n/2-1})$.
There exists an ambient metric for $(M,c)$ and it is unique up to addition 
of a term which is $O^+_{IJ}( \rho^{n/2})$ and up to pullback by a
diffeomorphism defined on a dilation-invariant 
neighborhood of 
$\cG\times\R$ which commutes with dilations and which restricts to the
identity on $\cG\times\{0\}$.  

The diffeomorphism invariance of ambient metrics can be broken by putting
them into a normal form with respect to a choice of metric $g$ in the
conformal class.  Observe first that the choice of $g\in c$ determines a 
trivialization of the bundle $\cG\rightarrow M$ by identifying 
$(t,x)\in \R_+\times M$ with $(x,t^2g_x)\in \cG$.  Under this
identification the tautological tensor ${\bf g}_0$ takes the form
${\bf g}_0=t^2 g$, where we omit writing $\pi^*$ for the pullback of a
tensor on $M$ to $\cG$, and we have $T=t\pa_t$.  There is an induced   
identification $\cG\times \R \cong \R_+\times M\times \R$.
A pre-ambient metric $\gt$ is said to be in normal form 
with respect to $g\in c$ if it satisfies the following three conditions: 
\begin{enumerate}
\item[(1)] Its domain of definition $\cGt$ has the
property that for each $z\in \cG$, the set of $\rho\in \R$ such that 
$(z,\rho)\in\cGt$ is an open interval $I_z$ containing $0$;
\item[(2)] For each $z\in \cG$, the parametrized curve
$I_z\ni \rho\mapsto (z,\rho)$ is a geodesic for $\gt$; 
\item[(3)] Under the identification 
$\cG\times \R\cong \R_+\times M \times \R$ induced by $g$, at each point 
$(t,x,0)\in \cG\times \{0\}$, $\gt$ takes the form 
$\gt =\, t^2 g +2tdtd\rho\,.$
\end{enumerate}
A straight
pre-ambient metric is in normal form with respect to $g$ if and only if it
has the form
\begin{equation}\label{normalform}
\gt=2tdtd\rho +2\rho\,dt^2+t^2g_\rho
\end{equation}
relative to the identification $\cG\times\R\cong\R_+\times M\times \R$
induced by $g$, 
where $g_\rho$ is a smooth family of metrics on $M$ parametrized by $\rho$
satisfying $g_0=g$.  Any pre-ambient metric can be put into normal form
with respect to a choice 
of $g\in c$ by a unique diffeomorphism which commutes with the dilations
and restricts to the identity on $\cG\times \{0\}$.  
For $n$ odd, the existence and uniqueness assertion 
for ambient metrics in normal form states that given a metric $g$ on $M$,  
there exists an ambient metric $\gt$ for $(M,[g])$ in normal form with
respect to $g$, and $g_\rho$ in \eqref{normalform} is uniquely determined
to infinite 
order at $\rho=0$.  For $n$ even, the corresponding assertion is that
$g_\rho$ is uniquely determined mod $O(\rho^{n/2})$ and also 
$\tr_g\left(\pa_\rho^{n/2}g_\rho|_{\rho=0}\right)$ is determined.  In all
dimensions $n\geq 3$ one has 
\begin{equation}\label{initial}
g_\rho = g + 2P\rho +O(\rho^2)
\end{equation}
where $P$ denotes the Schouten tensor of $g$, defined by
$$
(n-2)P_{ij}=R_{ij}-\frac{R}{2(n-1)}g_{ij}.
$$

For $n$ even a conformally invariant tensor, the ambient obstruction 
tensor, obstructs the existence of smooth solutions to
$\Ric(\gt)=O(\rho^{n/2})$.  However if the obstruction tensor vanishes then
there are smooth solutions to higher order.  (In general there are 
higher-order solutions with expansions involving log terms; see Theorem
3.10 of \cite{FG2}.)  If $(M,c)$ is a 
conformal manifold of even dimension $n\geq 4$, by an infinite-order
ambient metric we will mean a smooth straight  
pre-ambient metric for which $\Ric(\gt)$ vanishes to infinite order at
$\rho=0$.  If $(M,c)$ admits an infinite-order ambient metric, then it has
vanishing obstruction tensor.  The Taylor expansion of an infinite-order
ambient metric in normal form is no longer   
determined solely by the initial metric $g$.  It follows from Theorem 3.10
of \cite{FG2} that there is a natural pseudo-Riemannian invariant 1-form  
$D_i(g)$ depending on a metric $g$, so that if $g$ has vanishing
obstruction 
tensor and $\ka$ is a smooth symmetric 2-tensor on $M$ which is trace-free
with respect to $g$ and satisfies $\ka_{ij},^j=D_i(g)$ where the divergence
is with respect to the Levi-Civita connection of $g$, then there is an  
infinite-order  
ambient metric in normal form with respect to $g$ such that
$\tf\left(\pa_\rho^{n/2}g_\rho|_{\rho=0}\right)=\ka$.  Here $\tf$ denotes the
trace-free part with respect to $g$.  Moreover, these conditions uniquely
determine $g_\rho$ to infinite order at $\rho =0$ and all infinite-order  
ambient metrics in normal form relative to $g$ arise from such a 
$\ka$.  We will call $\ka$ the ambiguity in the infinite-order ambient   
metric.  

We will use capital Latin indices to label objects on $\cG\times\R$.  
Upon choosing a metric $g\in c$ we have the splitting $\cG\times\R\cong
\R_+\times M\times\R$.  We will use a $0$ index for the $\R_+$
factor, lower case Latin indices for the $M$ factor, and an $\infty$ index
for the $\R$ factor.  

For the purposes of this paper it will be convenient to define the tractor
bundle and connection in ambient terms.  Such a formulation was given in  
\cite{CG1}, \cite{BG} where further discussion and details may be found. 

Let $(M,c)$ be a conformal manifold with metric bundle
$\cG\stackrel{\pi}{\rightarrow} M$.  For $x\in M$, write
$\cG_x=\pi^{-1}(\{x\})$ for the fiber of $\cG$ over $x$.  
Recall that the bundle $\cD(w)$ of conformal densities of weight $w\in \C$
has fiber 
$\cD_x(w)=\{f:\cG_x\rightarrow \C:(\delta_s)^*f=s^wf,\;s>0$\}, so that  
sections of $\cD(w)$ on $M$ are functions on $\cG$ homogeneous of
degree $w$. 
A metric $g$ in the conformal class is a section of $\cG$, so if $f$ is a
section of $\cD(w)$, then 
$f\circ g$ is a function on $M$.  Under conformal change $\gh=\Omega^2g$,
we have $f\circ \gh = \Omega^wf\circ g$.  

The standard tractor bundle of $(M,c)$ and its metric and connection can be 
similarly defined in terms of homogeneous vector fields on $\cG_x$. 
Identify $\cG$ with the subset $\cG\times \{0\}\subset
\cG\times \R$ via the map $\iota$.  
Let $\gt$ be an ambient metric for $(M,c)$ defined on 
a dilation-invariant open neighborhood $\cGt$ of $\cG$ in $\cG\times\R$. 
Consider the rank $n+2$ vector bundle $\cT\rightarrow M$ with fiber     
\begin{equation}\label{tracdef}
\cT_x=\left\{U\in\Gamma(T\cGt\,\big{|}_{\cG_x}):
(\delta_s)_*U=sU,\;s>0\right\}.     
\end{equation}
So a section of $\cT$ on $M$ is the same as a section $U$ of
$T\cGt\,\big{|}_\cG$ on $\cG$ satisfying $(\delta_s)_*U=sU$, or
equivalently $(\delta_s)^*U=s^{-1}U$.  If $U$, $W\in \cT_x$, 
then $\gt(U,W)$ is
homogeneous of degree 0 on $\cG_x$, i.e. $\gt(U,W)\in \R$.  This therefore
defines a metric $h$ of   
signature $(p+1,q+1)$ on $\cT$.  Since $T$ is homogeneous of degree 0 with
respect to the $\delta_s$, it defines a section of 
$\cT(1):=\cT\otimes \cD(1)$.  But the set of $U$ which at each point of
$\cG_x$ is a 
multiple of $T$ constitutes a subbundle of $\cT$ which we denote 
$\text{span}\{T\}$.  Its orthogonal complement $\text{span}\{T\}^\perp$ is 
the set of $U$ which at each point of $\cG_x$ is tangent 
to $\cG$.  This gives the filtration
\begin{equation}\label{filtration}
0\subset \text{span}\{T\}\subset \text{span}\{T\}^\perp\subset \cT.
\end{equation}

In order to realize the tractor connection, observe that 
$\pi_*:T\cG\rightarrow TM$ induces a realization of the 
tangent bundle $TM$ as
$$
T_xM=\left\{V\in\Gamma(T\cG\,\big{|}_{\cG_x}):
(\delta_s)_*V=V,\;s>0\right\}\Big{/} \text{span}\{T\},  
$$
where now $\text{span}\{T\}$ really means the constant multiples of $T$.
If $v\in T_xM$, choose $V\in\Gamma(T\cG\,\big{|}_{\cG_x})$ representing
$v$.  If $U$ is a 
section of $\cT$ near $x$, define the tractor connection $\nabla$ by
$\nabla_vU = \nt_VU$.  Observe first that the right-hand side makes sense
since $U\in \Gamma(T\cGt\,\big{|}_\cG)$ and $V$ is tangent to $\cG$.  To
see that  
the right-hand side is independent of the choice of $V$ representing $v$ it
suffices to show that $\nt_TU=0$.  Now
$$
\nt_TU=\nt_UT+[T,U]=\nt_UT+\cL_TU
$$
where $\cL$ denotes the Lie derivative.  But $\cL_TU=-U$ by the homogeneity
of $U$ and $\nt_UT=U$ since $\gt$ is straight.  Thus $\nt_TU=0$ as 
desired.  Finally, 
$\nt_VU$ has the same homogeneity as $U$ since $V$ and $\nt$ are invariant
under the $\delta_s$:  $V$ is invariant by hypothesis and $\nt$ is 
invariant since $\delta_s$ is a homothety of $\gt$.  Thus $\nt_VU$ is a
well-defined section of $\cT$ and it is easily checked that this defines a
connection.  The tractor metric is parallel with respect to $\nabla$ since
$\nt\gt=0$.  But the filtration \eqref{filtration} is not
$\nabla$-parallel.  

Note that this realization of the tractor bundle and connection depends on 
the choice of ambient metric $\gt$.  If $\gt$ is changed by a
diffeomorphism, one obtains different but equivalent realizations.  

It is shown in \cite{CG1} that this formulation of the standard tractor
bundle and connection agrees with other definitions by using a functorial 
characterization of tractor bundles.  For our purposes it will be 
useful to see this in terms of the splitting induced by a     
choice of $g$.  As we did with conformal densities, we associate to  
$U\in \Gamma(\cT)$ the map $U\circ g$ defined on $M$, for which 
$(U\circ g)(x)\in T_{(x,g_x)}\cGt$.  We decompose 
$T_{(x,g_x)}\cGt$ via the splitting in which $\gt$ is
in normal form with respect to $g$.  That is, after composing with a 
diffeomorphism if necessary, we assume that $\gt$ is in normal form with
respect to $g$.  The splitting $\cG\times\R\cong
\R_+\times M\times\R$ induces a splitting 
$T(\cG\times\R)\cong T\R_+\oplus TM\oplus T\R$.  Via the trivializations of 
$T\R_+$ and 
$T\R$ induced by $\pa_t$ and $\pa_\rho$, resp., $U\circ g$ is expressed 
in the form $(U^0,U^i,U^\infty)$, where $U^0$ and $U^\infty$
are functions on $M$ and $U^i$ is a vector field on $M$.  In terms of
coordinates $(t,x,\rho)$ with respect to which $\gt$ is in normal form,
we are simply writing $U\circ g = U^0\pa_t +U^i\pa_{x^i}+U^\infty\pa_\rho$.   
In this way, relative to the choice of $g$ we represent a section 
$U\in\Gamma(\cT)$ by the triple $(U^0,U^i,U^\infty)$.  Recalling that
$t=1$ 
on points of $\cG$ of the form $(x,g_x)$, it is evident from condition (3)
in the definition of normal form that the tractor metric is given by 
$h(U,U) = 2U^0U^\infty  + g_{ij}U^iU^j$.   

Suppose now we make a conformal change to $\gh=e^{2\Up}g$.  
The ambient metric $\gt$ can be put into normal form relative to $\gh$ by
pulling back by a homogeneous diffeomorphism.  Arguing as in the proof of
Proposition 6.5 of \cite{FG2}, one can identify the Jacobian on $\cG$ of
the diffeomorphism  and thus calculate the relation between the
representation $(\wh{U}^0,\wh{U}^i,\wh{U}^\infty)$ of $U$ 
with respect to $\gh$ and that with respect to $g$.  The result is: 
$$
\begin{pmatrix}
\wh{U}^0\\
\wh{U}^i\\
\wh{U}^\infty
\end{pmatrix}
=
\begin{pmatrix}
e^{-\Up}&0&0\\
0&e^{-\Up}&0\\
0&0&e^{\Up}
\end{pmatrix}
\begin{pmatrix}
1&-\Up_j&-\tfrac12 \Up_k\Up^k\\
0&\delta^i{}_j&\Up^i\\
0&0&1
\end{pmatrix}
\begin{pmatrix}
U^0\\
U^j\\
U^\infty
\end{pmatrix}.
$$
This is the identification used in the construction of the
standard tractor bundle in \cite{BEG}, so shows the agreement of these 
constructions. 

The tractor connection can also be expressed in terms of the splitting.  
It is straightforward to calculate the Christoffel symbols of a metric of
the form \eqref{normalform}.  One obtains:
\begin{gather}\label{cnr}
\begin{gathered}
\Gat_{IJ}^0 = 
\left(
\begin{matrix}
0&0&0\\
0&-\frac12 tg_{ij}'&0\\
0&0&0
\end{matrix}
\right)\\
\Gat_{IJ}^k = 
\left(
\begin{matrix}
0&t^{-1}\delta_j{}^k&0\\
t^{-1}\delta_i{}^k&\Ga_{ij}^k&\frac12 g^{kl}g_{il}'\\ 
0&\frac12 g^{kl}g_{jl}'&0 
\end{matrix}
\right)\\
\Gat_{IJ}^\nf = 
\left(
\begin{matrix}
0&0&t^{-1}\\
0&-g_{ij} +\rho g_{ij}'&0\\
t^{-1}&0&0
\end{matrix}
\right).
\end{gathered}
\end{gather}
(See (3.16) of \cite{FG2}.)  
Here all $g_{ij}$ and $g^{ij}$ refer to $g_\rho$, $'$ denotes 
$\pa_\rho$, $\Gamma_{ij}^k$ denotes the Christoffel symbol 
of the metric $g_\rho$ with $\rho$ fixed, and the blocks correspond to the
splittings $I\leftrightarrow (0,i,\nf)$, $J\leftrightarrow (0,j,\nf)$.  
In order to represent $\nabla_vU=\nt_VU$ in terms of the splitting with
respect to $g$, we evaluate \eqref{cnr} at $\rho=0$, $t=1$ and consider
only the $I=i$ components.  Recalling \eqref{initial}, this gives for the 
tractor Christoffel symbols:
\begin{equation}\label{tracchrist}
\begin{split}
\Gat_{iJ}^0 &= 
\left(
\begin{matrix}
0&-P_{ij}&0
\end{matrix}
\right)\\
\Gat_{iJ}^k &= 
\left(
\begin{matrix}
\delta_i{}^k&\Ga_{ij}^k&P_i{}^k
\end{matrix}
\right)\\
\Gat_{iJ}^\nf &= 
\left(
\begin{matrix}
0&-g_{ij}&0  
\end{matrix}
\right).
\end{split}
\end{equation}
The tractor covariant derivative is given by
$\nabla_iU^K=\pa_iU^K+\Gat_{iJ}^KU^J$, or equivalently
\begin{equation}\label{tracconn}
\nabla_i
\begin{pmatrix}
U^0\\
U^k\\
U^\infty
\end{pmatrix}
=
\begin{pmatrix}
\nabla_i U^0-P_{ij}U^j\\
\nabla_i U^k+\delta_i{}^kU^0+P_i{}^kU^\nf\\
\nabla_iU^\infty-U_i
\end{pmatrix}.
\end{equation}
On the right-hand side, $\nabla_i U^0$ and $\nabla_i U^\nf$ denote the
exterior derivative on functions and $\nabla_i U^k$ the Levi-Civita
connection of $g$ on vector fields on $M$.  Equation \eqref{tracconn} is
taken as the definition of the tractor connection in \cite{BEG}, so
the ambient construction produces the usual (normal) tractor connection.   

\section{Parallel Extension}\label{parext}
In this section we prove 
Theorem~\ref{main}.  We will be dealing primarily with cotractors.  The
realization dual to \eqref{tracdef} is  
$$
\cT^*_x=\left\{\eta\in\Gamma(T^*\cGt\,\big{|}_{\cG_x}):
(\delta_s)^*\eta=s\eta,\;s>0\right\}
$$
so that $\eta(U)$ is homogeneous of degree 0 on $\cG_x$.  Hence if $r\in
\N$, we realize sections of $\otimes^r\cT^*$ as   
$\chi\in\Gamma(\otimes^r T^*\cGt\,\big{|}_{\cG})$ satisfying  
$(\delta_s)^*\chi=s^r\chi$.  Any such $\chi$ satisfies $\nt_T\chi=0$ and
the cotractor connection is realized by $\nabla_v\chi=\nt_V\chi$
by analogy with the discussion for tractors in \S\ref{ambtrac}.  We use the
splitting for $\cT^*$ dual to the one above:  $\chi\in \cT^*$ is
represented with respect to $g\in c$ by 
$\chi=(\chi_0,\chi_i, \chi_\nf)$ 
if $\chi\circ g=\chi_0dt+\chi_idx^i+\chi_\nf d\rho$.

If $\chi\in\Gamma(\otimes^r\cT^*)$, we say that 
$\ct\in\Gamma(\otimes^rT^*\cGt)$ is an ambient extension of $\chi$ if
$\delta_s^*\ct=s^r\ct$ and $\ct|_{\cG}=\chi$.  Clearly if $\nt\ct=0$ then
restricting to $\rho=0$ and to differentiations
tangent to $\cG$, it follows that $\nabla\chi=0$.  
Theorem~\ref{main} asserts that any parallel tractor admits an ambient 
extension which is as parallel as one can expect given the
indeterminacy in the ambient metric.   

Before beginning the proof of Theorem~\ref{main}, observe that uniqueness
of the asserted extension is easy:  if $n$ is odd then 
$\ct$ is unique to infinite order and if   
$n$ is even then $\ct \mod O(\rho^{n/2})$ is uniquely determined.  This
follows just from the fact that $\nt_\nf\ct$ vanishes to the stated order
by successively differentiating with respect to $\rho$ at $\rho =0$ the
vanishing condition applied to the difference.  (See the proof of
Proposition~\ref{beyondcritical} below.)  
Note also that since $\nt\ct$ depends on first derivatives of $\gt$, 
if $n$ is even then the indeterminacy of $\gt$ at order $n/2$ enters into 
$\nt\ct$ at order $n/2$.  This is the subject of \S\ref{neven}.  

The proof of Theorem~\ref{main} uses properties of the covariant
derivatives of the curvature tensor of an ambient metric.  We denote the
curvature tensor of a pre-ambient metric by $\Rt$ and covariant
derivatives with respect to its connection by indices preceded by a comma.
Proposition 6.1 of  
\cite{FG2} asserts that the curvature tensor of any straight pre-ambient
metric satisfies for $r\geq 0$
\begin{gather}\label{curv}
\begin{gathered}
T^L\Rt_{IJKL,M_1\cdots M_r}= 
-\sum_{s=1}^r \Rt_{IJKM_s,M_1\cdots \widehat{M_s} \cdots M_r}\\
T^P\Rt_{IJKL,PM_{1}\cdots M_r}= 
-2\Rt_{IJKL,M_1\cdots M_r} - 
\sum_{s=1}^r \Rt_{IJKL,M_s M_{1}\cdots
  \widehat{M_s}\cdots M_r}.  
\end{gathered}
\end{gather}
The empty sum on the right-hand sides is interpreted as 0 in case $r=0$.  
Suppose next that $\gt$ is an ambient metric in normal form.  The indices
can be specialized according to the normal form splitting $I\leftrightarrow
(0,i,\nf)$.  
\begin{lemma}\label{divergence}
If $\gt$ is an ambient metric in normal form, then at $\rho=0$, $t=1$ we
have  
$$
(2k+1)\Rt_{IJA\infty,\,}{}_{\underbrace{\scriptstyle{\infty\cdots\infty}}_{k-1}}
=\Rt_{IJA}{}^p{}_{,p\underbrace{\scriptstyle{\infty\cdots\infty}}_{k-1}}.
$$
This holds for all $k\geq 1$ if $n$ is odd and for $IJA$ and $k$ satisfying  
$\|IJA\|+2k\leq n+1$ if $n$ is even, where $\|0\|=0$, $\|i\|=1$ for $1\leq
i\leq n$, $\|\nf\|=2$, and $\|IJA\|=\|I\|+\|J\| +\|A\|$.  
\end{lemma}
\begin{proof}
First assume that $n$ is odd.  The second Bianchi identity and the fact
that $\gt$ is Ricci flat to infinite order imply at $\rho =0$ 
$$
\gt^{PQ}\Rt_{IJAP,Q\underbrace{\scriptstyle{\infty\cdots\infty}}_{k-1}}=0
$$
for all $k$.  Write out the trace using the normalization of $\gt$ at
$\rho=0$, $t=1$ to obtain  
$$
\Rt_{IJA0,\nf\underbrace{\scriptstyle{\infty\cdots\infty}}_{k-1}}
+\Rt_{IJA\nf,0\underbrace{\scriptstyle{\infty\cdots\infty}}_{k-1}}
+g^{pq}\Rt_{IJAp,q\underbrace{\scriptstyle{\infty\cdots\infty}}_{k-1}}=0.
$$
Now \eqref{curv} shows that 
$$
\Rt_{IJA0,\,}{}_{\underbrace{\scriptstyle{\infty\cdots\infty}}_{k}}
=-k\Rt_{IJA\infty,\,}{}_{\underbrace{\scriptstyle{\infty\cdots\infty}}_{k-1}},
\qquad\quad
\Rt_{IJA\nf,0\underbrace{\scriptstyle{\infty\cdots\infty}}_{k-1}}
=-(k+1)\Rt_{IJA\infty,\,}{}_{\underbrace{\scriptstyle{\infty\cdots\infty}}_{k-1}}
$$
so the result follows. 

If $n$ is even, the hypothesis $\|IJA\|+2k\leq n+1$ guarantees that
the relevant Ricci curvature derivative component vanishes so that the same
argument applies.  See Proposition 6.4 of \cite{FG2}.   
\end{proof}

\bigskip
\noindent
{\it Proof of Theorem~\ref{main}.}
Choose a metric $g$ in the conformal class.  Put $\gt$ into normal form 
relative to $g$.  The hypothesis $\nabla \chi=0$ is equivalent to the
statement that $\nt_A\ct|_{\rho =0}=0$ for $A=0,a$ for any extension.   
Define $\ct$ by parallel 
translation of $\chi$ along the lines $\rho\mapsto (t,x,\rho)$.
Parallel translation commutes with the dilations so $\ct$ has the right
homogeneity.  
It suffices to show that $\nt\ct$ vanishes to the stated order.  The point
is that the system $\nt\ct=0$ is overdetermined, so it must be shown that
parallel translation in the $\rho$-direction preserves vanishing of the
tangential covariant derivatives.  This requires study of the commutation 
of ambient covariant derivatives to high order.  Note that if $(M,g)$ is
real-analytic, then any parallel tractor $\chi$ is real-analytic, so if
$\gt$ is real-analytic, then $\ct$ is real-analytic as well.  

We show first that 
$\ct_{\mathcal{I},\,\underbrace{\scriptstyle{\infty\cdots\infty}}_{k}}=0$
on $\cGt$ for $k\geq 1$, where  
$\mathcal{I}=I_1\cdots I_r$ is any list of $r$ indices.  This follows by 
induction on $k$.  The case $k=1$ is clear since $\ct$ was defined to be
parallel in the $\rho$-direction.  For $k>1$ write
$$
\ct_{\mathcal{I},}{}_{\underbrace{\scriptstyle{\infty\cdots\infty}}_{k}}
=\partial_\infty\ct_{\mathcal{I},}{}_{\underbrace{\scriptstyle{\infty\cdots\infty}}_{k-1}}
-\sum_{s=1}^r\Gt_{I_s\nf}^J
\ct_{I_1\cdots I_{s-1}JI_{s+1}\cdots
  I_r,}{}_{\underbrace{\scriptstyle{\infty\cdots\infty}}_{k-1}} 
-\sum_{l=1}^{k-1}\Gt_{\nf\nf}^A
\ct_{\mathcal{I},}{}_{\underbrace{\scriptstyle{\infty\cdots\nf}}_{l-1}
  A\underbrace{\scriptstyle{\infty\cdots\nf}}_{k-l-1}}.  
$$
The first two terms vanish by the induction hypothesis and the last because
$\Gt^A_{\nf\nf}=0$ from \eqref{cnr}.

We argue next that if $\nt^l\ct|_{\rho =0}=0$ for $1\leq l\leq k-1$,  
then $\nt^{k+1}\ct|_{\rho=0}$ is symmetric in the last $k$ of the $k+1$ 
differentiation indices.  (We think of the differentiation indices as
listed after the indices of $\ct$, separated by a comma.)   This follows by  
commuting two adjacent indices 
among the last $k$ in $\nt^{k+1}\ct|_{\rho=0}$.  This commutation of
derivatives can be written
by the differentiated Ricci identity as a sum of terms of the form
derivatives of 
curvature contracted into derivatives of $\ct$, in which the total number
of derivatives of $\ct$ is at least 1 and drops by at least 2.  Thus the
result follows. 

Now we proceed with the main induction: we prove by induction on $k$ the
statement $\nt^k\ct|_{\rho =0}=0$.  
For $k=1$, $\nt_A\ct|_{\rho=0}=0$ for $A = 0,a$ since $\chi$ 
was parallel, and $\nt_\infty\ct|_{\rho =0} =0$ since it was extended to
be parallel along the $\rho$ lines.   

Assume $\nt^l\ct|_{\rho =0}=0$ for $1\leq l\leq k$.  We must show that 
$\nt^{k+1}\ct|_{\rho =0}=0$ and can assume that $k\leq n/2-2$ if $n$ is
even.  We show that 
$\ct_{\mathcal{I},A_0\cdots A_{k}}|_{\rho =0}=0$ by considering various
cases for the indices.  First, if $A_{k}\neq \infty$, then the result
follows by expanding out
the last covariant derivative in terms of Christoffel symbols and using the
induction hypothesis and the fact that $\partial_{A_{k}}$ is tangent 
to $\{\rho =0\}$.  Next, if $A_l\neq \infty$ for some $l>0$, then we can
commute $A_l$ all the way to the right using the above observation about
symmetry in the last indices, and thus reduce to the case $A_{k}\neq
\infty$.  So we can assume all $A_l=\infty$ for $l>0$.  

If $A_0=\infty$, then our first observation above does the job.  
In the following, all quantities are understood to be evaluated at
$\rho=0$, $t=1$.  If $A_0\neq \infty$, write
\begin{equation}\label{commute}
\begin{split}
\ct_{\mathcal{I},A_0\underbrace{\scriptstyle{\infty\cdots\infty}}_{k}}
=&\ct_{\mathcal{I},\infty
  A_0\underbrace{\scriptstyle{\infty\cdots\infty}}_{k-1}} 
+\Big(\sum_{s=1}^r\Rt^J{}_{I_sA_0\infty}
\ct_{I_1\cdots I_{s-1}JI_{s+1}\cdots  I_r}\Big)
{}_{,\,\underbrace{\scriptstyle{\infty\cdots\infty}}_{k-1}} \\ 
=&\sum_{s=1}^r\Rt^J{}_{I_sA_0\infty,\,}{}_{\underbrace{\scriptstyle{\infty\cdots\infty}}_{k-1}}
\ct_{I_1\cdots I_{s-1}JI_{s+1}\cdots I_r},
\end{split}
\end{equation}
where for the second equality we have used the induction hypothesis and
the fact that
$\ct_{\mathcal{I},\infty
  A_0\underbrace{\scriptstyle{\infty\cdots\infty}}_{k-1}} =0$ since a
differentiation index after the first is not $\infty$. 
Now the first equation in \eqref{curv} and the skew-symmetry of the
curvature tensor in the second pair of indices imply 
$\Rt^J{}_{I0\infty,\,}{}_{\underbrace{\scriptstyle{\infty\cdots\infty}}_{k-1}}=0$ 
for all $J$, $I$ and $k\geq 1$, so we obtain
$\ct_{\mathcal{I},0\underbrace{\scriptstyle{\infty\cdots\infty}}_{k}}=0$. 

We have left only to consider
$\ct_{\mathcal{I},a\underbrace{\scriptstyle{\infty\cdots\infty}}_{k}}$
with $1\leq a\leq n$.  We
intend to apply Lemma~\ref{divergence} to 
$\Rt^J{}_{I_sa\infty,\,}{}_{\underbrace{\scriptstyle{\infty\cdots\infty}}_{k-1}}$
on the right-hand side of \eqref{commute} for $A_0=a$ after lowering the
index $J$.  We verify the hypothesis in
Lemma~\ref{divergence} in case $n$ is even.  Denoting by $L$ the index   
replacing $J$ when it is lowered 
and recalling that $k\leq n/2-2$, we have 
$\|LI_sa\|+2k \leq 5+(n-4)= n+1$.  Thus the application of  
Lemma~\ref{divergence} is justified.

Equation \eqref{commute}, Lemma~\ref{divergence} and the induction
hypothesis thus give 
\[
\begin{split}
(2k+1)\ct_{\mathcal{I},a\underbrace{\scriptstyle{\infty\cdots\infty}}_{k}}
=&(2k+1)\sum_{s=1}^r\Rt^J{}_{I_sa\infty,\,}{}_{\underbrace{\scriptstyle{\infty\cdots\infty}}_{k-1}}  
\ct_{I_1\cdots I_{s-1}JI_{s+1}\cdots I_r}\\
=&\sum_{s=1}^r\Rt^J{}_{I_sa}{}^p{}_{,p\underbrace{\scriptstyle{\infty\cdots\infty}}_{k-1}} 
\ct_{I_1\cdots I_{s-1}JI_{s+1}\cdots I_r}\\
=&
\Big(\sum_{s=1}^r\Rt^J{}_{I_sa}{}^p\ct_{I_1\cdots I_{s-1}JI_{s+1}\cdots
  I_r}\Big){}_{,\,p\underbrace{\scriptstyle{\infty\cdots\infty}}_{k-1}}\\
=&\big(\ct_{\mathcal{I},a}{}^p
-\ct_{\mathcal{I},}{}^p{}_a\big)_{,\,p\underbrace{\scriptstyle{\infty\cdots\infty}}_{k-1}}\\
=&\ct_{\mathcal{I},a}{}^p{}_{p\underbrace{\scriptstyle{\infty\cdots\infty}}_{k-1}}
-\ct_{\mathcal{I},}{}^p{}_{ap\underbrace{\scriptstyle{\infty\cdots\infty}}_{k-1}}\\
=&\ct_{\mathcal{I},a}{}^p{}_{\underbrace{\scriptstyle{\infty\cdots\infty}}_{k-1}p}
-\ct_{\mathcal{I},}{}^p{}_{a\underbrace{\scriptstyle{\infty\cdots\infty}}_{k-1}p},
\end{split}
\]
where the last equality uses the observation about symmetry in last 
differentiated indices.  Now write out each of the final covariant
differentiations in the last line in terms of Christoffel symbols.  Since
we have shown that  
$\ct_{\mathcal{I},A_0\cdots A_{k}}=0$ unless 
$A_1=\cdots =A_{k}=\infty$ and $1\leq A_0\leq n$, consulting with
\eqref{cnr} to evaluate the Christoffel symbols gives 
$$
\ct_{\mathcal{I},a}{}^p{}_{\underbrace{\scriptstyle{\infty\cdots\infty}}_{k-1}p}
=\partial_p\ct_{\mathcal{I},a}{}^p{}_{\underbrace{\scriptstyle{\infty\cdots\infty}}_{k-1}} 
+\Gt_{p0}^p\ct_{\mathcal{I},a}{}^0{}_{\underbrace{\scriptstyle{\infty\cdots\infty}}_{k-1}} 
=n\ct_{\mathcal{I},a\underbrace{\scriptstyle{\infty\cdots\infty}}_{k}} 
$$
and 
$$
\ct_{\mathcal{I},}{}^p{}_{a\underbrace{\scriptstyle{\infty\cdots\infty}}_{k-1}p}
=\partial_p\ct_{\mathcal{I},}{}^p{}_{a\underbrace{\scriptstyle{\infty\cdots\infty}}_{k-1}} 
-\Gt_{pa}^\infty\ct_{\mathcal{I},}{}^p{}_{\underbrace{\scriptstyle{\infty\cdots\infty}}_{k}}
=\ct_{\mathcal{I},a\underbrace{\scriptstyle{\infty\cdots\infty}}_{k}}. 
$$
Thus we obtain 
$(2k+1)\ct_{\mathcal{I},a\underbrace{\scriptstyle{\infty\cdots\infty}}_{k}}
=(n-1)\ct_{\mathcal{I},a\underbrace{\scriptstyle{\infty\cdots\infty}}_{k}}$,
or
$$
(2k+2-n)\ct_{\mathcal{I},a\underbrace{\scriptstyle{\infty\cdots\infty}}_{k}}=0.   
$$
Hence the induction proceeds to all orders if $n$ is odd, but we must
impose $k\neq n/2-1$ if $n$ is even.    
\stopthm

A consequence of the proof is the following proposition.
\begin{proposition}\label{beyondcritical}
Let $n$ be even and suppose that $\gt$ is an infinite-order ambient metric
for $(M,c)$.  (In particular $(M,c)$ has vanishing obstruction tensor.)
Let $\chi\in\Gamma\left(\otimes^r\mathcal{T}^*\right)$ satisfy
$\nabla\chi =0$.  If $\chi$ has an ambient extension satisfying
$\nt\ct=O(\rho^{n/2})$, then $\chi$ has an ambient extension satisfying 
$\nt\ct=O(\rho^\nf)$.  
\end{proposition}
\begin{proof}
Let $\ct$ be the extension obtained by parallel translation along the lines 
$\rho\mapsto (t,x,\rho)$ as in the proof of Theorem~\ref{main} and let 
$\ct'$ be an extension satisfying $\nt\ct'=O(\rho^{n/2})$.  Then 
$\ct-\ct'=O(\rho)$ and $\nt_\nf(\ct-\ct')=O(\rho^{n/2})$.  Writing the
latter equation in terms of $\pa_\rho$ and Christoffel symbols and then
successively 
differentiating with respect to $\rho$ at $\rho=0$ shows that  
$\ct-\ct'=O(\rho^{n/2+1})$.  Hence $\nt\ct=O(\rho^{n/2})$.  
Since the induction in the proof of Theorem~\ref{main} requires only $k
\neq n/2 - 1$, it shows that
$\nt\ct=O(\rho^\nf)$.  Lemma~\ref{divergence} which is used in the proof
holds for $\gt$ for all $k$ 
because it just depends on homogeneity, straightness, and vanishing of the
Ricci curvature.     
\end{proof}

We remark that Theorem~\ref{main} implies a large collection of
integrability conditions on a parallel  
$\chi\in\Gamma\left(\otimes^r\mathcal{T}^*\right)$.  
First suppose $n$ is odd. Let $\ct$ be an
ambient extension of $\chi$ as in Theorem~\ref{main}.  Commuting covariant
derivatives shows that $\Rt.\ct=O(\rho^\nf)$, where $\Rt.$ denotes the
action of  
$\Rt$ viewed as an element of $\Gamma(\Lambda^2T^*\cGt\otimes\End T\cGt)$, 
so that $\Rt.\ct \in \Gamma(\Lambda^2T^*\cGt\otimes\otimes^rT^*\cGt)$.
Further 
differentiating this equation and restricting to $\rho =0$ shows that 
$\big((\nt^k\Rt)|_\cG\big).\chi=0$, where now $(\nt^k\Rt)|_\cG\in \Gamma  
\big((\otimes^kT^*\cGt\otimes
\Lambda^2T^*\cGt\otimes\End T\cGt)\big|_{\cG}\big)$ acts on $\chi$ via
the $\End T\cGt$ factor.  We can     
regard $(\nt^k\Rt)|_{\cG}$ as a section of a weighted tensor power of a
tractor bundle (see Proposition 6.5 of \cite{FG2}), so 
can regard $\big((\nt^k\Rt)|_\cG\big).\chi=0$ as a purely tractor equation
which holds as a 
consequence of the fact that $\chi$ is parallel.  The special case when
$k=0$ and the free $\Lambda^2T^*\cGt$ indices on $\Rt$ are tangent 
to $\cG$ recovers the fact that the tractor curvature annihilates $\chi$.   
When $k=0$ and these free indices are $j\nf$, one obtains relations
involving contractions of the Cotton and Bach tensors into $\chi$.
Relations of these types for special cases may be found in the literature,  
for example for $r=1$ in \cite{Go1} and for $\chi$ skew-symmetric in
\cite{Leit1}, \cite{Leit2}.  In 
general, Theorem~\ref{main} is equivalent to proving the integrability
conditions $\big((\nt^k\Rt)|_\cG\big).\chi=0$ for all $k$.  A similar   
discussion holds for $n$ even, keeping in consideration that the number of 
differentations transverse to $\cG$ is restricted.

\section{Critical order for $n$ even}\label{neven}

In this section we consider some examples illustrating what can happen
concerning ambient extension of parallel tractors at the critical order  
in even dimensions.  The main issue is  
that in general an ambient metric is no longer determined to higher order
solely by the conformal structure; there is an ambiguity at order $n/2$.
So whether or not a given parallel 
tractor has a parallel ambient extension can depend on which ambient metric
is chosen.  The natural general framework 
for such investigations would be to consider ambient metrics with log terms
as in 
Theorem 3.10 of \cite{FG2}.  In this case parallel extensions of tractors
must be expected to have expansions with log terms as well.  For simplicity
we generally restrict consideration here to the case of conformal
structures with vanishing obstruction tensor, for which log terms do not
enter.  The one exception is that 
our discussion of Fefferman metrics of nondegenerate hypersurfaces in
$\C^n$ applies also to some ambient metrics with log terms.

The dependence of the covariant derivatives of the curvature tensor of
an ambient metric on the  
ambiguity plays an important role in these considerations.  Let $n\geq 4$
be even, let 
$(M,c)$ be a conformal structure with vanishing obstruction tensor, and let  
$\gt$ be an infinite-order ambient metric for $(M,c)$ in normal form
relative to a representative $g\in c$.  Recall from \S\ref{ambtrac}
that $\gt$ takes the form \eqref{normalform} and 
$\tf\left(\pa_\rho^{n/2}g_\rho|_{\rho=0}\right)$ is the ambiguity.  
Define the strength of lists of
indices $\|IJ\cdots K\|$ as in Lemma~\ref{divergence}.  Proposition 6.2 of
\cite{FG2} shows that if $\|IJKLM_1\cdots M_r\|\leq n+1$, then the
restriction to $\rho=0$, $t=1$ of $\Rt_{IJKL,M_1\cdots M_r}$ is independent
of the ambiguity and defines a natural tensor invariant of $g$ as the
indices between $1$ and $n$ vary; in fact it  
can be expressed as a linear combination of contractions 
$\text{contr}(\nabla^{m_1}R\otimes \cdots \otimes\nabla^{m_N}R)$ of the 
covariant derivatives of the curvature tensor of $g$ such that  
$2N+2+\sum_{i=1}^Nm_i=\|IJKLM_1\cdots M_r\|$. 
On the other hand, if 
$\|IJKLM_1\cdots M_r\|\geq n+2$, then $\Rt_{IJKL,M_1\cdots
  M_r}|_{\rho=0,\,t=1}$ may depend 
on the ambiguity.  Proposition 6.6 of \cite{FG2} shows that the component    
$\Rt_{\infty ij\infty,\,\underbrace{\scriptstyle{\infty\cdots\infty}}_{(n-4)/2}}$ 
parametrizes the ambiguity:  any two infinite-order ambient metrics in
normal form for
which this component agree must agree to infinite order, and this component
can 
be prescribed to be an arbitrary trace-free symmetric 2-tensor 
on $M$ subject to the condition that its divergence is a particular natural 
1-form invariant of the initial metric $g$.  Another treatment of similar
considerations concerning the ambiguity is contained in \cite{CG1}.   

Let now $\chi\in \Gamma(\otimes^r\cT^*)$ be parallel for $(M,c)$.
We are interested in the questions of existence and uniqueness of
infinite-order ambient 
metrics $\gt$ for which $\chi$ has an ambient extension  
satisfying $\nt\ct=O(\rho^\nf)$.  Proposition~\ref{beyondcritical} shows
that for a given $\gt$, $\chi$ has such an extension if and only if it has
an extension satisfying $\nt\ct=O(\rho^{n/2})$.  We will see that 
for some $\chi$ there is no such extension for any choice of $\gt$,
for some $\chi$ there is such an extension for precisely one $\gt$, 
and for some $\chi$ such extensions exist for infinitely many choices of
$\gt$.    

A trivial example of a parallel tractor in even dimensions which has a
parallel 
ambient extension for more than one infinite-order ambient metric is the
tractor metric.  Clearly for any ambient metric, the ambient metric  
itself provides a parallel extension.  If one index is raised, the 
identity endomorphism of the tractor bundle has as a parallel extension the
identity endomorphism of the ambient tangent bundle, and in this
realization the parallel extension is actually independent of which 
ambient metric is chosen.  

We begin by formulating a condition on $\chi$ under which there is at most
one choice of infinite-order ambient metric with respect to which $\chi$
has an ambient extension satisfying $\nt\ct=O(\rho^{n/2})$.   
Recall that a 1-form on $M$ can be inserted as
the injecting part of an adjoint tractor.  (An adjoint tractor is a section 
of $\End\cT$ which is skew with respect to $h$.)  Namely, if $\eta\in
T_x^*M$ then $\pi^*\eta$ defines a section of  
$T^*\mathcal{G}|_{\cG_x}$ which annihilates $T$.  
Since $T_I$ spans the annihilator of $T\cG\subset T\cGt$, we may regard   
$\pi^*\eta$ as an equivalence class of sections of 
$T^*\mathcal{\widetilde{G}}|_{\cG_x}$ defined modulo $T_I$ and homogeneous
of degree 0 with respect to the $\delta_s$, and 
therefore as an equivalence class of elements of
$\cT^*_x(-1)$   
defined modulo $T_I$.  So we can define a bundle map $\cI:T^*M\rightarrow   
\End_{\text{skew}}\cT $ by
$\cI(\eta)^J{}_{I}=2h^{JK}T_{[K}(\pi^*\eta)_{I]}$, where $h^{JK}$ denotes
the inverse tractor metric.  With respect to the splitting induced by any 
representative $g$ one has $\cI(\eta)^0{}_i=\eta_i$,
$\cI(\eta)^j{}_\infty=-\eta^j$, other $\cI(\eta)^J{}_I=0$.  
Now let $\chi\in\Gamma\left(\otimes^r\mathcal{T}^*\right)$.
Define a bundle map 
$F_{\chi}:T^*M\rightarrow \otimes^r\cT^*$ by 
$F_\chi(\eta)= \cI(\eta). \chi$, where the . refers to the
action of $\End\cT $ on $\otimes^r\cT^*$.  

\begin{definition}
Let $\chi\in\Gamma\left(\otimes^r\mathcal{T}^*\right)$.
We say that $\chi$ is {\it determining} if the induced
map $F_\chi:\Gamma(T^*M)\rightarrow \Gamma(\otimes^r\cT^*)$ is injective.
\end{definition}

Note that we are requiring that the induced map on smooth sections be
injective, not that $F_\chi$ is injective as a bundle map.
Observe that the tractor metric $h$ is not determining since $K.h=0$    
for any adjoint tractor $K$.  

\begin{proposition}\label{determining}
Let $(M,c)$ be a conformal manifold of even dimension $n\geq 4$ with
vanishing obstruction tensor.  Let  
$\chi\in\Gamma\left(\otimes^r\mathcal{T}^*\right)$ 
be parallel and determining.  Then up to infinite order and up to
diffeomorphism, 
there is at most one infinite-order ambient metric with respect to which
$\chi$ has an ambient extension satisfying  $\nt\ct = O(\rho^{n/2})$.   
\end{proposition}
\begin{proof}
Pick $g\in c$ and assume that $\gt$ is in normal form with respect to $g$.
If $\ct$ is an extension satisfying $\nt\ct=O(\rho^{n/2})$, then at $\rho
=0$, $t=1$ we have (notation as in the proof of Theorem~\ref{main}):
\begin{equation}\label{determine}
\begin{split}
0=&\ct_{\mathcal{I},a\underbrace{\scriptstyle{\infty\cdots\infty}}_{(n-2)/2}}
-\ct_{\mathcal{I},\infty
  a\underbrace{\scriptstyle{\infty\cdots\infty}}_{(n-4)/2}}\\
=&\Big(\sum_{s=1}^r\Rt^J{}_{I_sa\infty}
\ct_{I_1\cdots I_{s-1}JI_{s+1}\cdots  I_r}\Big)
{}_{,\,\underbrace{\scriptstyle{\infty\cdots\infty}}_{(n-4)/2}} \\ 
=&\sum_{s=1}^r\Rt^J{}_{I_sa\infty,\,\underbrace{\scriptstyle{\infty\cdots\infty}}_{(n-4)/2}} 
\ct_{I_1\cdots I_{s-1}JI_{s+1}\cdots I_r}.
\end{split}
\end{equation}
Of the components
$\Rt^J{}_{Ia\infty,\,\underbrace{\scriptstyle{\infty\cdots\infty}}_{(n-4)/2}}$
which enter this equation, only 
$\Rt^0{}_{ia\infty,\,\underbrace{\scriptstyle{\infty\cdots\infty}}_{(n-4)/2}}$
and 
$\Rt^j{}_{\infty
  a\infty,\,\underbrace{\scriptstyle{\infty\cdots\infty}}_{(n-4)/2}}$ 
depend on the ambiguity of the ambient metric.  All other components are
determined by $g$.  Choose a vector 
field $v$ on $M$.  Define a 1-form $\eta$ on $M$ by 
$\eta_i=v^a\Rt_{\infty 
  ia\infty,\,\underbrace{\scriptstyle{\infty\cdots\infty}}_{(n-4)/2}}$.  
Then 
$\cI(\eta)^J{}_I
=v^a\Rt^J{}_{Ia\infty,\,\underbrace{\scriptstyle{\infty\cdots\infty}}_{(n-4)/2}}$  
if ${}^J{}_I={}^0{}_i$ or ${}^J{}_I={}^j{}_\infty$ and 
$\cI(\eta)^J{}_I=0$ otherwise.  If we set 
$D^J{}_{I}=v^a\Rt^J{}_{Ia\infty,\,\underbrace{\scriptstyle{\infty\cdots\infty}}_{(n-4)/2}} 
-\cI(\eta)^J{}_{I}$, then $D^J{}_{I}$ is independent of the ambiguity in the
ambient metric.  Now the contraction of $v^a$ with \eqref{determine} can be 
written 
$F_\chi(\eta)= -D.\chi$.  By the injectivity of $F_\chi$ on
$\Gamma(T^*M)$, it follows that for each $v$ there is at most one
possibility for $\eta$.  Therefore there is at most one possibility for the
tensor 
$\Rt_{\infty
  ia\infty,\,\underbrace{\scriptstyle{\infty\cdots\infty}}_{(n-4)/2}}$. 
But this tensor determines the ambiguity in the ambient metric. 
\end{proof}

Let us consider the case $r=1$ so that $\chi$ is a parallel section of
$\cT^*$, which we assume is nonzero.  We sometimes call such $\chi$ a
parallel tractor 1-form.  In the terminology of Gover \cite{Go3},  
$\chi$ defines an almost Einstein structure.  There is a large 
literature concerning parallel tractor 1-forms, especially because of their
relation to 
conformally Einstein metrics and their basic role in conformal holonomy. 

Upon writing the equation $\nabla\chi=0$ in the splitting   
corresponding to a metric $g$ using \eqref{tracchrist}, one finds that
$\chi$ is determined by $\chi_0$, and that $\chi_0$ must satisfy the
equation $\tf\big((\nabla^2_{ij}+P_{ij})\chi_0\big)=0$.  Specifically,
writing $\sigma = \chi_0$, $\chi$ is given by  
$$
\chi=
\begin{pmatrix}
\chi_0\\
\chi_i\\
\chi_\infty
\end{pmatrix}
=
\begin{pmatrix}
\sigma\\
\sigma_i\\
-\tfrac{1}{n}(\Delta \sigma + J\sigma)  
\end{pmatrix}
$$
where $\Delta=\nabla^k\nabla_k$ and $J=R/2(n-1)$. 
In this way parallel sections of $\cT^*$ are in one-to-one-correspondence 
with solutions to the conformally invariant equation
\begin{equation}\label{tfP}
\tf\big((\nabla^2_{ij}+P_{ij})\sigma\big)=0.
\end{equation}  
Since $\chi$ is parallel
and nonzero it is nonvanishing, so the 2-jet of $\sigma$ is nonvanishing. 
In particular, 
$\{\sigma\neq 0\}$ is open and dense.  The conformal 
transformation law for the trace-free Schouten tensor shows that on this
set the rescaled metric $\sigma^{-2}g$ is Einstein.  If
$\Sigma:=\{\sigma=0\}$ is nonempty, the conformal 
structure extends smoothly across $\Sigma$ but the Einstein 
representative $\sigma^{-2}g$ becomes singular.  Since the obstruction
tensor vanishes for conformal classes containing an Einstein metric, by
continuity it vanishes for even-dimensional conformal classes admitting a
nonzero parallel tractor 1-form.  

\begin{proposition}\label{rank1}
Let $\chi\in\Gamma(\cT^*)$ be nonzero and satisfy $\nabla\chi =0$.  
Then $\chi$ is determining.
\end{proposition}
\begin{proof}
We have $\big(F_\chi(\eta)\big)_I=-\cI(\eta)^J{}_I\chi_J$.  In particular,
$\big(F_\chi(\eta)\big)_i=-\eta_i\chi_0$.  
Since $\chi_0\neq 0$ on an open dense set, it follows that $F_{\chi}$ is 
injective on $\Gamma(T^*M)$.  
\end{proof}

Theorem~\ref{main} shows that if $n$ is odd, then any parallel tractor
1-form has an ambient extension parallel to infinite order (with respect to 
the unique ambient metric up to infinite order and diffeomorphism), and the
same is true to order $n/2-1$ if 
$n$ is even.  Propositions~\ref{determining} and
\ref{rank1} imply that for $n$ even, there is at most one determination 
of the 
ambiguity in an infinite-order ambient metric with respect to which $\chi$
has an 
ambient extension parallel to infinite order.  We analyze the question in
this case of whether there exists a choice of the ambiguity with respect to
which $\chi$ has an ambient extension parallel to infinite order.  

Consider first the situation near the set $\{\sigma\neq 0\}$; equivalently 
consider a conformal class containing a chosen Einstein metric.  
In this case in all dimensions one can identify explicitly an ambient  
metric for which there is an explicit parallel extension of $\chi$.  If $g$
is Einstein and we set   
$\lambda = J/n=R/2n(n-1)$, then $\gt$ defined by \eqref{normalform} with  
$g_\rho = (1+\lambda \rho)^2g$ satisfies $\Ric(\gt)=0$.  If $n$ is odd then
this is to infinite order the unique ambient metric in normal form relative
to $g$.  However, if $n$ is even then there are other infinite-order
ambient metrics corresponding to
other choices of the ambiguity.  Proposition 7.5 of \cite{FG2} shows that
for $n$ even this choice is canonical in the sense that
up to infinite order and up to diffeomorphism, it is independent of which
Einstein metric in the conformal class is chosen (usually there is only one
up to homothety, but there may be others).  Henceforth, if $g$ is     
Einstein we will denote by $\gt^c$ the ambient metric given
by \eqref{normalform} with 
\begin{equation}\label{ambcan}
g_\rho = (1+\lambda \rho)^2g.
\end{equation}

If $g$ is Einstein, the corresponding solution of \eqref{tfP} is
$\sigma=1$.  Thus the parallel tractor is given in the associated Einstein 
scale by $\chi=(1,0,-\lambda)$.  
It is straightforward to verify using \eqref{cnr} with  
$g_\rho = (1+\lambda \rho)^2g$ that   
a parallel extension of $\chi$ for the ambient metric $\gt^c$ is given by   
$\ct=(1-\lambda \rho)dt-t\lambda\,d\rho=d[t(1-\lambda\rho)]$.  
We conclude
\begin{proposition}\label{canonicalparallel}
Let $n\geq 4$ be even and suppose that $(M,c)$ is a conformal class
containing an 
Einstein metric $g$.  Up to infinite order and up to diffeomorphism, the 
canonical ambient metric $\gt^c$ is the unique infinite-order ambient
metric for $(M,c)$ with respect to which the 
associated parallel tractor 1-form $\chi$ has an ambient extension
satisfying $\nt\ct=O(\rho^{n/2})$.
\end{proposition}

We remark that \cite{Leit1} for $\lambda\neq 0$ and \cite{Leis} for
$\lambda=0$ have shown that $\gt^c$ can be written in a way that makes it
evident that any parallel tractor has a parallel ambient extension (in fact
even that the holonomy of the tractor connection equals the holonomy of
$\gt^c$). Namely, if $\lambda\neq 0$, set $u=(1+\la\rho)t$,   
$v=(1-\la\rho)t$.  Then 
$$
\gt^c=2d(\rho t)dt +t^2(1+\la\rho)^2g = \frac{1}{2\la}(du^2-dv^2)+u^2g
$$
is translation-invariant in $v$ relative to the $(u,v,x)$ coordinates and 
$\pa_v$ is parallel.  Thus a parallel extension of any parallel 
tractor is obtained by extending it to be translation-invariant in $v$.
When $\la=0$, the same reasoning applies upon setting $v=\rho t$ so that  
$\gt^c=2dvdt+t^2g$ is translation-invariant in $v$ relative to $(t,v,x)$
coordinates and $\partial_v$ is parallel.  When combined with
Proposition~\ref{determining}, this 
gives another proof that for $n$ even, the canonical ambient metrics
associated to two different Einstein metrics in the same conformal class
are related by a diffeomorphism to infinite order (Proposition 7.5 of
\cite{FG2}).  Namely, by the Leitner/Leistner result, the parallel tractor
associated to the second 
Einstein metric has a parallel ambient extension relative to the canonical
ambient metric for the first Einstein metric, so by
Proposition~\ref{determining}, the two canonical
ambient metrics are diffeomorphic to infinite order.  

A basic class of almost Einstein structures for which $\Sigma$   
is nonempty consists of those determined by Poincar\'e-Einstein metrics.
We will give necessary  
and sufficient conditions on an even-dimensional Poincar\'e-Einstein metric
in order that the associated parallel tractor 1-form admit   
an ambient extension satisfying $\nt\ct=O(\rho^{n/2})$ for some choice of
infinite-order ambient metric.  In particular,
this will give examples of parallel tractor 1-forms which do not admit such
an ambient extension for any infinite-order ambient metric.    

Let $\Sigma\subset M$ be an embedded hypersurface.
In this paper we will say that a metric $g_+$ on
$M\setminus \Sigma$ is 
Poincar\'e-Einstein if $\Ric(g_+)=-(n-1)g_+$ and if
$g_+=r^{-2}g$ for some 
defining function $r$ for $\Sigma$ and smooth metric $g$ on $M$ of
signature $(p,q)$ such that $g|_{T\Sigma}$ has signature $(p-1,q)$.  (In
particular, since $g$ is assumed smooth, the obstruction tensor of 
$g|_{T\Sigma}$ vanishes if $n\geq 5$ is odd.)  We could alternately take
$g_+$ to satisfy $\Ric(g_+)=(n-1)g_+$ and require $g|_{T\Sigma}$ to have
signature $(p,q-1)$; the two formulations are equivalent under the change 
$g_+\mapsto -g_+$, $(p,q)\mapsto (q,p)$.  

The conformal structure $(M,c)$ with $c=[g]$ is   
determined by $g_+$ and admits a parallel tractor 1-form given relative to
$g$ by    
\begin{equation}\label{tractor1form}
\chi=
\begin{pmatrix}
\chi_0\\
\chi_i\\
\chi_\infty
\end{pmatrix}
=
\begin{pmatrix}
r\\
r_i\\
-\tfrac{1}{n}(\Delta r + Jr)
\end{pmatrix}.
\end{equation}
Possibly rescaling $r$ and $g$, one can identify a neighborhood of $\Sigma$ 
in $M$ with a neighborhood of $\Sigma\times\{0\}$ in $\Sigma\times \R$ 
so that $r$ is the variable in $\R$ and  
$g$ takes the form $g=dr^2+h_r$ for a smooth 1-parameter family $h_r$ 
of metrics on $\Sigma$.  If $n$ is odd, then the Taylor expansion of $h_r$
in $r$ is even to infinite order.  If $n\geq 4$ is even, then $h_r$ is even 
in $r$ to order $n-2$, but a typical Poincar\'e-Einstein metric has    
$\partial_r^{n-1}h_r|_\Sigma\neq 0$.  (See \cite{FG2}.) 
For $n$ even, we say that $g_+$ is 
even if $\partial_r^{n-1}h_r|_\Sigma=0$.  This is independent of the
choices, and the Taylor expansion of $h_r$ is then even to infinite order.
If $n$ is odd, all Poincar\'e-Einstein metrics will be said to be 
even.  
\begin{proposition}\label{evenpe}
Let $n\geq 4$ be even and let $g_+$ be a Poincar\'e-Einstein metric with
associated parallel tractor $\chi$.  There exists an infinite-order ambient
metric for $(M,c)$ with respect to 
which $\chi$ has an ambient extension satisfying $\nt\ct=O(\rho^{n/2})$ if
and only if $g_+$ is even.  
\end{proposition}
\begin{proof}
Write $g=dr^2+h_r$ near $\Sigma$ as above.  Note first that since 
$|dr|^2_g=1$, the 
conformal transformation law $P_{g_+}=P_g+r^{-1}\nabla^2r -\frac12
|dr|^2_g\,g_+$ of the 
Schouten tensor and the Einstein condition on $g_+$ imply that  
$\nabla^2 r+P_gr=0$.  Taking the trace gives $\Delta r+Jr=0$, so
$\chi_\nf=0$.      

Let $\gt$ be an infinite-order ambient metric for $(M,c)$ in normal  
form relative to $g$ and let $\ct$ be the extension given in the 
proof 
of Theorem~\ref{main} which satisfies $\nt\ct=O(\rho^{(n-2)/2})$.  The
first steps in the induction in Theorem~\ref{main} hold also at the next
order, so we know that 
$\ct_{I,\infty
  a\underbrace{\scriptstyle{\infty\cdots\infty}}_{(n-4)/2}}=0$
at $\rho =0$, $t=1$, and that $\chi$ has an extension 
satisfying $\nt\ct=O(\rho^{n/2})$ if and only if  
$\ct_{I,a\underbrace{\scriptstyle{\infty\cdots\infty}}_{(n-2)/2}}=0$   
at $\rho =0$, $t=1$.  All succeeding expressions are understood to 
be evaluated at $\rho =0$, $t=1$.    Equation \eqref{commute} 
thus shows that $\chi$ has an extension  
satisfying $\nt\ct=O(\rho^{n/2})$ if and only if  
$\Rt^J{}_{Ia\infty,\,\underbrace{\scriptstyle{\infty\cdots\infty}}_{(n-4)/2}}\ct_J
=0$; that is if and only if (recall $\chi_\nf=0$)
\begin{equation}\label{ambcurv}
\Rt_{\infty Ia\infty,\,\underbrace{\scriptstyle{\infty\cdots\infty}}_{(n-4)/2}}r
=-\Rt^j{}_{Ia\infty,\,\underbrace{\scriptstyle{\infty\cdots\infty}}_{(n-4)/2}}r_j.
\end{equation}

The right-hand side of \eqref{ambcurv} for $I=i$ is
independent of the ambiguity 
of the ambient metric; it is determined by $g$ alone.  Using the 
expression for the leading terms in the ambient metric expansion 
(see (3.21) of \cite{FG2}), it is not hard to see that 
the tensor 
$\Rt_{jia\infty,\,\underbrace{\scriptstyle{\infty\cdots\infty}}_{(n-4)/2}}$
which appears on the right-hand side of \eqref{ambcurv} 
for $I=i$ has the form 
\begin{equation}\label{deltacotton}
\Rt_{jia\infty,\,\underbrace{\scriptstyle{\infty\cdots\infty}}_{(n-4)/2}}
=c\Delta^{(n-4)/2}C_{aij} + \Lambda_{aij},
\end{equation}
where $C_{aij}=P_{ai,j}-P_{aj,i}$ denotes the Cotton tensor of $g$, $c\neq
0$, and 
$\Lambda_{aij}$ is a linear combination of contractions of covariant
derivatives of curvature of $g$ which are at least quadratic and involve at
most $n-5$ derivatives of curvature. 

Label objects on $\Sigma$ by Greek letters $\al$, $\be$ and let 
a $\ds$ index correspond to $r$ in the identification $M\sim \Sigma\times
\R$, so that $i\leftrightarrow (\al,\ds)$.  Denote $\partial_r$ by $'$.    
We first show that the right-hand side of \eqref{ambcurv} for 
$I=\be$ and $a=\al$ vanishes on $\Sigma$ if and only if $g_+$ is even.     
By \eqref{deltacotton} we have
$$
\Rt^j{}_{\be\al\infty,\,\underbrace{\scriptstyle{\infty\cdots\infty}}_{(n-4)/2}}r_j
=\Rt_{\ds\be\al\infty,\,\underbrace{\scriptstyle{\infty\cdots\infty}}_{(n-4)/2}}
=c\Delta^{(n-4)/2}C_{\al\be\ds} + \Lambda_{\al\be\ds}.
$$
Now $\Lambda_{\al\be\ds}$ can be written as a 
linear combination of contractions 
$\text{contr}(\nabla^{m_1}R\otimes \cdots \otimes\nabla^{m_N}R)$ 
with three free indices $\al$, $\be$, $\ds$.  The contractions can be 
expanded corresponding to the block diagonal form of $g$.  In each term, 
at least one of the $\nabla^{m_i}R$ will have an odd number of $\ds$
indices.  Since $h_r$ is even to order $n-2$ and $\Lambda$ involves at most
$n-3$ derivatives of $g$, this component $\nabla^{m_i}R$ vanishes on
$\Sigma$.  Thus $\Lambda_{\al\be\ds}=0$ on $\Sigma$.  
Similar analysis considering also the leading nonzero term shows that   
$$
\Delta^{(n-4)/2}C_{\al\be
  \diamondsuit}=c'\pa_r^{n-1}h_{\al\be}\quad\text{on  }\Sigma
$$
where $c'\neq 0$.  Thus the right-hand side of \eqref{ambcurv}
for $I=\be$, $a=\al$ 
vanishes on $\Sigma$ if and only if $\pa_r^{n-1}h_{\al\be}=0$ on 
$\Sigma$, that is if and only if $g_+$ is even.  

It follows immediately that if $\chi$ has an 
ambient extension satisfying $\nt\ct=O(\rho^{n/2})$ then $g_+$ is even,
since the left-hand side of \eqref{ambcurv} clearly vanishes on $\Sigma$.  

Next we prove the converse. 
Away from $\Sigma$, $\chi$ is the parallel    
tractor associated to the Einstein metric $g_+$.  
We have seen in this case in Proposition~\ref{canonicalparallel} that  
there is a unique
determination of the ambiguity in the ambient metric so that $\chi$ has an
ambient extension satisfying $\nt\ct=O(\rho^{n/2})$.  Thus away from
$\Sigma$ there is a unique determination of 
$\Rt_{\infty
  ia\infty,\,\underbrace{\scriptstyle{\infty\cdots\infty}}_{(n-4)/2}}$ 
for which \eqref{ambcurv} holds.  
We will show that if $g_+$ is even, then the so-determined tensor
$\Rt_{\infty
  ia\infty,\,\underbrace{\scriptstyle{\infty\cdots\infty}}_{(n-4)/2}}$ 
extends smoothly across $\Sigma$.  The desired result follows immediately.  
This tensor is trace-free and satisfies the divergence constraint on
$\Sigma$ relative to $g$ by continuity, so it determines an infinite-order
ambient metric. All quantities appearing in \eqref{ambcurv} 
are then smooth across $\Sigma$, so \eqref{ambcurv} holds on $\Sigma$ by
taking limits.  (If $g_+$ is not even, the tensor 
$\Rt_{\infty
  ia\infty,\,\underbrace{\scriptstyle{\infty\cdots\infty}}_{(n-4)/2}}$
uniquely determined away from $\Sigma$ blows up on approach to $\Sigma$.)    

To see that 
$\Rt_{\infty
  ia\infty,\,\underbrace{\scriptstyle{\infty\cdots\infty}}_{(n-4)/2}}$
extends smoothly across $\Sigma$, first 
take $I=\infty$ in \eqref{ambcurv} to obtain
$\Rt_{\infty\ds
  a\infty,\,\underbrace{\scriptstyle{\infty\cdots\infty}}_{(n-4)/2}}=0$.
So the components 
$\Rt_{\infty
  ia\infty,\,\underbrace{\scriptstyle{\infty\cdots\infty}}_{(n-4)/2}}$ 
with $i$ or $a=\ds$ certainly extend smoothly across $\Sigma$; they
vanish identically nearby.  We have seen that if $g_+$ is even, 
then the right 
hand side of \eqref{ambcurv} for $I=\be$, $a=\al$ extends smoothly across
$\Sigma$ and vanishes on $\Sigma$.  Thus dividing by $r$ shows that 
$\Rt_{\infty
  \be\al\infty,\,\underbrace{\scriptstyle{\infty\cdots\infty}}_{(n-4)/2}}$ 
extends smoothly across $\Sigma$ as well.  
\end{proof}

If $n$ is even and $g_+$ is an even Poincar\'e-Einstein metric, then
the ambient metric of Proposition~\ref{evenpe} with respect to which $\chi$
has an ambient extension parallel to infinite order
is uniquely determined to infinite order 
up to diffeomorphism by Propositions~\ref{determining} and \ref{rank1}.
It may therefore be regarded as a distinguished ambient metric for the
conformal class $c$ determined by $g_+$.
Proposition~\ref{canonicalparallel} shows that away 
from $\Sigma$ it 
agrees to infinite order up to diffeomorphism with the canonical ambient
metric determined by the Einstein metric $g_+
\in c|_{M\setminus\Sigma}$ in the sense of our previous discussion.    

We conclude our discussion of Poincar\'e-Einstein metrics by showing   
that in all dimensions, one can identify explicitly the parallel ambient 
extension of the parallel tractor 1-form \eqref{tractor1form} corresponding
to an even Poincar\'e-Einstein metric (as well as the ambient metric with
respect to which it is parallel). 
This is closely related to Theorem 6.3 of \cite{Go3} and uses a   
formula of Juhl \cite{J} for a Poincar\'e-Einstein metric whose conformal
infinity is $g=dr^2+h_r$, where $g_+=r^{-2}g$ is an 
even Poincar\'e-Einstein metric.  Recall that we defined a
Poincar\'e-Einstein metric to be even 
if $g_+=r^{-2}(dr^2+h_r)$ where the Taylor expansion of $h_r$ is even in
$r$ to infinite order, and this holds for all Poincar\'e-Einstein metrics
if $n$ is odd.  If $g_+$ is even and positive definite, Biquard's  
unique continuation theorem (\cite{B}) shows that $h_r=h_{-r}$ for $r$ near 
$0$.  But we do not know if this holds for other signatures.  For
simplicity of exposition, in the subsequent discussion we will restrict 
attention to even Poincar\'e-Einstein metrics for which $h_r=h_{-r}$.  
This is no real loss of generality:  one can construct two such metrics
from any even Poincar\'e-Einstein metric by reflecting across $r=0$ the
restriction of $h_r$ to either $r> 0$ or $r< 0$.  

Let $n\geq 3$ and let $g_+=r^{-2}g$ be an even Poincar\'e-Einstein metric
with $g=dr^2+h_r$ where $h_r=h_{-r}$.  Define $k_u$ for $u\geq 0$ small by 
$k_u=h_{\sqrt{u}}$ so that $k_u$ is smooth up to $u=0$ and $h_r=k_{r^2}$
for all $r$ near $0$.  Juhl has discovered (Theorem 7.2 of \cite{J}) that
the metric $g_{++}$ given by     
\begin{equation}\label{doublepe}
g_{++}=s^{-2}\big(ds^2+dr^2+k_{r^2+s^2}\big)
\end{equation}
is a Poincar\'e-Einstein metric with conformal infinity $[g]$ in 
normal form relative to $g$.   
Juhl proves this by a direct calculation that $g_{++}$ satisfies 
$\Ric(g_{++})=-ng_{++}$ for $s\neq 0$.
An alternate proof (which could be used to guess the result) is to begin 
with the Einstein metric $g_+$ in the   
conformal class $[g]$ away from $\Sigma$.  The formula for the
``canonical'' 
Poincar\'e metric associated to an Einstein metric (see (7.13) of
\cite{FG2}; this is the Poincar\'e metric analogue of \eqref{ambcan} above)
gives 
\begin{equation}\label{peagain}
g_{++}=s^{-2}\left[ds^2+(1+\tfrac14 s^2)^2g_+\right].
\end{equation}
Now $g_{++}$ can be put into normal form relative to the metric
$g=dr^2+h_r$ in the 
conformal class.  This can be carried out explicitly:  
substituting $g_+=r^{-2}(dr^2+k_{r^2})$ in \eqref{peagain} 
and making the change of variables
$$
r=\sqrt{r'^2+s'^2}\qquad\qquad s=2\frac{\sqrt{r'^2+s'^2}-r'}{s'}
$$
with inverse
$$
r'=\frac{4-s^2}{4+s^2}\;r\qquad\qquad s'=\frac{sr}{1+\tfrac14 s^2}
$$
one obtains without difficulty
$g_{++}=s'^{-2}\big(ds'^2+dr'^2+k_{r'^2+s'^2}\big)$.  Relabeling the
variables gives \eqref{doublepe}.  

The fact that $g_+$ is even is of course not used in verifying that
$g_{++}$ defined  
by \eqref{doublepe} is Einstein.  But if $g_+$ is not even, then $g_{++}$
has as conformal infinity $[dr^2+h_{|r|}]$, which is not smooth across
$\{r=0\}$ and does not agree with $[g]$ for $r<0$.  

Recall that in normal form, an ambient metric and the associated even
Poincar\'e metric 
are related by the change of variable $s^2=-2\rho$; see Chapter 4  
of \cite{FG2}.  Thus the metric $\gt_+$ defined by \eqref{normalform} with
$g_\rho=dr^2+k_{r^2-2\rho}$ satisfies $\Ric(\gt_+)=0$ for $\rho<0$.  (In
fact, this holds where $g_\rho$ makes sense, i.e. for $2\rho<r^2$.)  If we 
choose some smooth extension of $k_u$ to $\{u<0\}$ and define $g_\rho$ by
the same formula, then $\gt_+$ is defined in a neighborhood of $\{\rho=0\}$  
and is an ambient metric (to infinite order in any dimension) for $[g]$ in 
normal form relative to $g$.
\begin{proposition}\label{peextension}
Suppose $n\geq 3$ and let $g_+=r^{-2}g$ be an even Poincar\'e-Einstein
metric, where $g=dr^2+k_{r^2}$.  Let $\chi$ be the associated parallel
tractor 1-form given by \eqref{tractor1form}.  Let $\gt_+$ be the ambient
metric defined 
by \eqref{normalform} with $g_\rho=dr^2+k_{r^2-2\rho}$ as above.  Then the 
1-form $\ct=rdt+tdr=d(rt)$ is a parallel ambient extension of $\chi$.  
($\ct$ is independent of $\rho$ and has zero $d\rho$ component.)  
\end{proposition}
\begin{proof}
We have observed at the beginning of the proof of Proposition~\ref{evenpe}
that $\chi_\nf=0$, so $\ct$ is certainly an 
ambient extension of $\chi$.  The condition $\nt\ct=0$ becomes
$\nt^2(rt)=0$ and is a straightforward  
verification from \eqref{cnr}; the function $rt$ is a quadratic polynomial 
in the coordinates.  The verification uses only that $g_\rho$ has the form     
$dr^2+k_{r^2-2\rho}$ for some 1-parameter family of metrics $k$ on
$\Sigma$; the Einstein condition does not enter.
\end{proof}

Proposition~\ref{peextension} of course gives an alternate proof 
of the existence of an ambient extension of $\chi$
parallel to infinite order (Theorem~\ref{main} and
Proposition~\ref{evenpe}) for even Poincar\'e-Einstein metrics.  It also
identifies explicitly the distinguished ambient metric associated to  
$g_+$ defined immediately after the proof of Proposition~\ref{evenpe}.  
Finally, it gives an explicit realization of Theorem 6.3 of \cite{Go3}:      
by the relation again between ambient metrics and Poincar\'e metrics in
normal form, the ambient metric in normal form relative to the metric
$h_0=g|_{T\Sigma}$ on $\Sigma$ is 
$\gt=2\rho dt^2+2tdtd\rho +t^2k_{-2\rho}=\gt_+|_{r=0,dr=0}$.  

Consider next the case of tractor 2-forms.  Again there is much literature;
see e.g. \cite{Leit1}, \cite{Leit2}, \cite{Go2}, \cite{H}.  A section 
$\chi\in\Gamma(\Lambda^2\cT^*)$ has 
components $\chi_{0j}\in \Gamma(T^*M)$, $\chi_{ij}\in
\Gamma(\Lambda^2T^*M)$, $\chi_{0\nf}\in C^\nf(M)$, $\chi_{j\nf}\in
\Gamma(T^*M)$ with   
respect to the splitting determined by a choice of representative 
$g$.   Just as for tractor 1-forms, writing the equation $\nabla  
\chi=0$ in terms of components using the tractor Christoffel symbols
\eqref{tracchrist} 
shows that a parallel $\chi$ is determined by its projecting part
$\chi_{0j}$, and the projecting part satisfies the conformal Killing
equation.   
Set $\alpha_j=\chi_{j0}=-\chi_{0j}$.  If $\nabla \chi=0$, then $\alpha$
satisfies $\alpha_{(i,j)}=\frac{1}{n}\alpha_{k,}{}^kg_{ij}$ and one has 
\begin{equation}\label{2form}
\chi_{ij}=\alpha_{[i,j]},\qquad \chi_{0\infty}=\tfrac{1}{n}\alpha_{k,}{}^k,
\qquad \chi_{j\nf}=\tfrac{1}{n}\alpha_{k,}{}^k{}_j+P_j{}^k\alpha_k. 
\end{equation}
If $\chi$ is not identically $0$ then it is nonvanishing, so the 2-jet of 
$\alpha$ is nonvanishing and $\{\alpha\neq 0\}$ is dense.    
\begin{proposition}\label{rank2}
Let $\chi\in\Gamma(\Lambda^2\cT^*)$.  If 
$\alpha_k\alpha^k\neq 0$ on a dense set, then $\chi$ is 
determining.    
\end{proposition}
\begin{proof}
We have
$F_\chi(\eta)_{IJ}=-\cI(\eta)^K{}_I\chi_{KJ}-\cI(\eta)^K{}_J\chi_{IK}$, so 
$F_\chi(\eta)_{ij}=2\eta_{[i}\alpha_{j]}$.  Thus if 
$F_\chi(\eta)=0$, then on the set where $\alpha\neq 0$ we have
$\eta=c\alpha$ for a smooth function $c$.    
Now $F_\chi(\eta)_{0\infty}=0$ gives $\eta^k\alpha_{k}=0$, so it follows
that $c=0$ where $\alpha^k\alpha_k\neq 0$.  Hence $\eta=0$ on a dense 
set, so $\eta=0$ everywhere since it is smooth.
\end{proof}
\noindent
In particular, if the signature is definite, every nonzero parallel tractor 
2-form is determining.  

On the other hand, for the null case we have
\begin{proposition}\label{rank2null}
Let $\chi\in\Gamma(\Lambda^2\cT^*)$ satisfy $\nabla\chi =0$.  If 
$\alpha_k\alpha^k$ vanishes identically, then $\chi$ is not determining.       
\end{proposition}
\begin{proof}
We can assume that $\alpha$ is not identically zero.  It suffices to show
that $F_\chi(\alpha)=0$.  We have $F_\chi(\eta)_{j0}=0$ for any $\eta$.
The argument in the proof of 
Proposition~\ref{rank2} shows that $F_\chi(\alpha)_{ij}=0$ and  
$F_\chi(\alpha)_{0\infty}=0$.  It remains to show that 
$F_\chi(\alpha)_{j\infty}=0$.  Using \eqref{2form} we have 
$$
-F_\chi(\alpha)_{j\infty}=\alpha_j\chi_{0\infty}-\alpha^k\chi_{jk}
=\tfrac{1}{n}\alpha_j\alpha_{k,}{}^k-\alpha^k\alpha_{[j,k]}
=\tfrac{1}{n}\alpha_j\alpha_{k,}{}^k-\tfrac12 \alpha^k\alpha_{j,k}
+\tfrac12 \alpha^k\alpha_{k,j}.  
$$
Now $\alpha^k\alpha_{k,j}=0$ since $\alpha$ is null.  Contracting the
conformal Killing equation with $\alpha$ then gives
$\alpha^k\alpha_{j,k}=\frac{2}{n}\alpha_j\alpha_{k,}{}^k$, so   
$F_\chi(\alpha)_{j\infty}=0$ as desired.
\end{proof}

Proposition~\ref{rank2null} suggests that parallel tractor 2-forms with
null projecting part are candidates for existence of parallel   
ambient extensions for more than one infinite-order ambient metric.  We
show next that this happens for Fefferman conformal 
structures associated to nondegenerate hypersurfaces in $\C^n$.  

In \cite{Leit3} and \cite{CG2}, it was shown that that Fefferman conformal   
structures of nondegenerate integrable CR manifolds of hypersurface type
are locally 
characterized by the existence of a parallel almost complex structure 
$\J^I{}_J$ on $\cT$ such that $\chi_{IJ}:=\J_{IJ}$ is a tractor  
2-form, where the index is lowered with the tractor metric.  Equivalently
they are characterized by the existence of a parallel tractor 2-form $\chi$
such that raising an index gives an almost complex structure.  The
underlying manifold $M$ of the  
Fefferman conformal structure is a circle bundle over the CR manifold 
and $\J$ is determined by the property that the vector field on $M$  
given by $\J T \mod T$ is the infinitesimal $S^1$ action.  This makes sense
because $\J T$ is a section of $T\cGt|_\cG$ homogeneous of degree $0$
and orthogonal to $T$ with respect to the tractor metric, so it projects to 
a vector field on $M$.  The fact that $\J$ is
determined by this projection is just the statement that $\chi$ is
determined by its projecting part $\alpha$, which we have seen in
\eqref{2form}.  

In Fefferman's original construction the CR manifold is a nondegenerate 
real hypersurface $\cM$ in $\C^n$, $n\geq 2$, in which case the circle
bundle $M=S^1\times  
\cM$ is trivial.  Fefferman showed in \cite{F} that there is a smooth
defining function $u$ for $\cM$ uniquely determined mod $O(u^{n+2})$ such
that $J(u)=1+O(u^{n+1})$, where  
$$
J(u)= (-1)^{r+1}\det
\begin{pmatrix}
u&u_{\bar{j}}\\
u_i&u_{i\bar{j}}
\end{pmatrix}_{1\leq i,j \leq n}.
$$
That $\cM$ is nondegenerate means that the Levi form
$-u_{i\bar{j}}|_{T^{1,0}\cM}$ is nondegenerate, and we have taken its
signature to be $(r,s)$, 
$r+s=n-1$.  Set $\C^*=\C\setminus \{0\}$.  A representative $g$ for the
conformal structure is by definition the pullback to    
$$
S^1\times
\cM = \{(z^0,z): |z^0|=1,\,z\in \cM\}\subset \C^*\times \C^n
$$ 
of the  K\"ahler metric $\gt$ defined on a neighborhood $\cGt$ of
$\C^*\times \cM$ in $\C^*\times \C^n$  by  
\begin{equation}\label{crambient}
\gt=\pa^2_{\al{\overline \be}}(-|z^0|^2u)dz^\al d{\overline z}^\be.
\end{equation}
The metric bundle $\cG$ can be identified with $\C^*\times \cM$ and the
dilations on $\cGt$ are just the usual dilations on the $\C^*$ factor.  
Now $\gt$ is clearly homogeneous of degree 2 and it satisfies the initial
condition $\iota^*\gt={\bf g}_0$ by the definition of $g$.  It is easily
checked that $\gt$ is straight.  Its Ricci curvature   
\begin{equation}\label{ricci}
\pa^2_{\al{\overline \be}}(\log |\det g_{\rho{\overline \sigma}}|)
dz^\al d{\overline z}^\be 
=\pa^2_{i{\overline j}}(\log J(u))dz^i d{\overline z}^j
\end{equation}
is clearly $O(u^{n-1})$ and is easily seen to be $O^+_{IJ} (\rho^{n-1})$.  
So $\gt$ is an ambient metric for $(S^1\times \cM, [g])$.  

Fix a defining function $u$ 
satisfying $J(u)=1+O(u^{n+1})$.  Theorem 2.11 of 
\cite{Gr2} shows that for each $a\in C^\nf(\cM)$, there is a $v$ uniquely 
determined to infinite order having an asymptotic expansion of the form  
\begin{equation}\label{vseries}
v\sim u\sum_{k=0}^\nf\eta_k(u^{n+1}\log|u|)^k  
\end{equation}
with each $\eta_k$ smooth, 
such that $\eta_0=1+au^{n+1} +O(u^{n+2})$ and $J(v)=1$ to infinite
order. 
For such $v$, $\gt$ defined by \eqref{crambient} with $u$ replaced by $v$
has $\Ric(\gt)=0$ to 
infinite order by \eqref{ricci}.  We will call such a $\gt$ an
infinite-order ambient metric  
with log terms for $(S^1\times \cM, [g])$.  These are parametrized by the
scalar ambiguity $a$.    
\begin{proposition}\label{crextension}
Let $\cM\subset \C^n$ be a nondegenerate hypersurface with associated
Fefferman conformal structure $(S^1\times \cM, [g])$ and parallel tractor
2-form $\chi\in \Gamma(\Lambda^2\cT^*)$.  For each infinite-order ambient 
metric $\gt$ with log terms parametrized by $a\in C^\infty(\cM)$ as above,
$\chi$ has a parallel ambient extension.    
\end{proposition}
\begin{proof}
Let $\Jt$ denote the almost complex structure on $\C^*\times \C^n$.  If
$\gt$ is any one of the infinite-order ambient 
metrics with log terms, then $\gt_{IK}\Jt^K{}_J$ is skew and $\nt\Jt=0$ 
since $\gt$ is K\"ahler.  
Now $\Jt$ is invariant under the dilations, so   
$\Jt|_\cG$ defines an almost complex structure on $\cT$ (recall that 
$\cG$ is 
identified with $\C^*\times\cM$) which we denote by $\J$.  We claim that 
$\J$ is the parallel almost complex structure of \cite{Leit3}, \cite{CG2}.   
Restricting the properties for $\Jt$ shows that $h_{IK}\J^K{}_J$ is skew
and $\J$ is parallel.  If we write $z^0=re^{i\theta}$ in polar     
coordinates, then $\J(r\pa_r)=\pa_\theta$.  But $r\pa_r=T$ is the 
infinitesimal dilation and $\pa_\theta$ the conformal Killing field giving  
the infinitesimal $S^1$ action, so this is the required condition on $\J$. 
It follows that $\J$ is the almost complex structure of \cite{Leit3},
\cite{CG2}.  
Lowering an index, we conclude that 
$\ct_{IJ}=\gt_{IK}\Jt^K{}_J$ is a parallel extension of 
$\chi_{IJ}=h_{IK}\J^K{}_J$.  
\end{proof}

Note that the Einstein condition is not used in the proof of  
Proposition~\ref{crextension}; what matters is that $\gt$ is K\"ahler.  In
particular, Proposition~\ref{crextension} holds for all metrics of the form
\eqref{crambient} so long as $J(u)=1+O(u^2)$. 
But    
if $\cM$, $u$ and $a$ are real-analytic, then the series \eqref{vseries}  
converges (\cite{Ki}) and thus defines a real-analytic function off $\cM$.
The 
corresponding metric $\gt$ is then Ricci-flat and K\"ahler.  The
$(n+1,0)$-form $(z^0)^ndz^0\wedge dz^1\wedge \cdots \wedge dz^n$ is
parallel, so 
$\Hol(\gt)\subset SU(r+1,s+1)$.  For any choice of   
real-analytic $a$, the almost complex structure and the complex volume form 
have thus been simultaneously extended to be parallel.  In this regard it
is interesting 
to recall that a simply connected conformal structure with conformal
holonomy contained in $U(p,q)$ 
must have holonomy contained in $SU(p,q)$ (\cite{Leit3}), but 
this is not the case for metric holonomy.  Thus for conformal structures
the existence of a (local) parallel complex volume form follows from the
existence of a parallel $\J$, but the existence of a parallel extension of
$\J$ does not imply the existence of a parallel extension of the complex
volume form.  

Observe the similarity between the situation in 
Proposition~\ref{crextension} and that for the trivial example of 
extending the tractor metric.  The extension $\ct_{IJ}$ is the 
K\"ahler form of $\gt$ so depends on the ambient metric which has been
chosen. But the extension $\ct^I{}_J=\Jt^I{}_J$ with an index raised is
independent of this choice.  Likewise the extended parallel complex volume  
form $(z^0)^ndz^0\wedge dz^1\wedge \cdots \wedge dz^n$ is independent of
choice of $\gt$.   

There is a scalar CR invariant $L$ defined to be a constant multiple of  
$(J(u)-1)/u^{n+1}|_\cM$ which is 
independent of the choice of $u$ satisfying $J(u)=1+O(u^{n+1})$.  
The main properties of $L$ are derived in \cite{Lee}, \cite{Gr1},
\cite{Gr2}.  Proposition 3.10 of \cite{GH} shows that the obstruction
tensor of $(S^1\times \cM, [g])$ is a constant multiple of $L\theta^2$,
where $\theta = \frac{i}{2}(\pa u-\bar{\pa} u )|_{T\cM}$ is the associated
pseudohermitian 1-form and the tensor $L\theta^2$ on $\cM$ is implicitly
pulled back to $S^1\times \cM$.  So $[g]$ is   
obstruction-flat if and only if $L=0$.  In this case Proposition 2.16 of 
\cite{Gr2} shows that for all choices of $a\in C^\nf(\cM)$ the coefficients
$\eta_k$ for $k\geq 1$ vanish to infinite order in the expansion of $v$
above.  The corresponding $\gt$ are therefore  
smooth and are infinite-order ambient metrics in our usual sense.  
Hence Proposition~\ref{crextension} asserts in particular that the parallel 
tractor 2-form $\chi$ for an obstruction-flat Fefferman conformal structure
of a nondegenerate hypersurface in $\C^n$ has a parallel ambient extension
for a family of infinite-order ambient metrics (without log terms)
parametrized by $a\in 
C^\nf(\cM)$.  Likewise, the discussion in the paragraph following 
Proposition~\ref{crextension} shows that if $\cM$ is obstruction-flat and
real-analytic and 
$a$ is chosen to be real-analytic, then the resulting metrics $\gt$ are
real-analytic everywhere and have holonomy contained in $SU(r+1,s+1)$.

\section{$G_2$ Holonomy}\label{G2holonomy}

In this section we apply Theorem~\ref{main} to Nurowski's conformal
structures associated to a generic 2-plane field on a 5-manifold.  Let 
$\cD$ be a distribution of rank 2 on a manifold $M$ of dimension 5.  We
assume that $M$ is connected and, for convenience, oriented (which is
equivalent to the condition that $\cD$ is oriented).  (Our arguments and
results all 
carry over to the nonorientable case upon replacing below $G_2$ by 
$\{\pm I\} G_2\subset O(3,4)$ and $P$ by 
$\{A\in \{\pm I\} G_2:Ae_0=\lambda e_0, \lambda>0\}$.)    
Let $X$, $Y$ be a local basis for the sections of $\cD$.  $\cD$ is
said to be generic if    
$X$, $Y$, $[X,Y]$, $[X,[X,Y]]$, $[Y,[X,Y]]$ are linearly independent at
each point.  In this case we denote by
$\cD^1=\operatorname{span}\{X,Y,[X,Y]\}$ the derived rank 3 distribution.     
In \cite{C}, Cartan solved the equivalence problem for such
distributions by canonically associating to $\cD$ a principal bundle 
$\cB\rightarrow M$ and 
Cartan connection $\omega$ taking values in the Lie algebra $\fg_2$ of
$G_2$, the split real form of the exceptional group.  The structure   
group of $\cB$ is the parabolic subgroup $P\subset G_2$ fixing a null ray.
In the appendix we fix our conventions,  
review the Cartan connection and derive two properties which will be 
needed in the sequel.  

In \cite{N1}, Nurowski showed that there is a conformal structure on $M$ of
signature $(2,3)$ 
naturally determined by $\cD$.  This follows from the existence 
of Cartan's canonical bundle and connection and the relation between the
relevant groups.  One way to see this is as follows.  It is a general   
fact that the tangent bundle $TM$ of a manifold $M$ with a Cartan geometry
of type $(\fg,P)$ is associated to the adjoint representation of $P$ on
$\fg/\fp$.  We claim that the adjoint
action of our $P$ on $\fg_2/\fp\cong \R^5$ preserves up to scale a
quadratic form of signature $(2,3)$.  This implies the existence
of a canonical conformal structure by the associated bundle construction.
Now $G_2$ is a subgroup of $SO(3,4)$.  Let $P_c$ be the subgroup of  
$SO(3,4)$ fixing the same null ray, with Lie algebra $\fp_c$.  Then
$P\subset P_c$ and $\fs\fo(3,4)/\fp_c\cong \fg_2/\fp$.  The adjoint
action of $P_c$ on $\fs\fo(3,4)/\fp_c$ preserves up to scale a
quadratic form of signature $(2,3)$; this is the reason a Cartan geometry
of type $(\fs\fo(3,4),P_c)$ induces a conformal structure.  Since the  
adjoint action of $P$ on $\fg_2/\fp$ can be regarded as the
restriction of the adjoint action of $P_c$, it preserves the same quadratic
form up to scale.  

As discussed in the introduction, any generic $\cD$ can be written locally
in the form \eqref{Dform} for some smooth function $F$ such that
$F_{qq}$ is nonvanishing, and $\cD$ defined by \eqref{Dform} is
generic for any such $F$.  The choice $F=q^2$ gives the homogeneous model
distribution.  For $\cD$ in this form, Nurowski gave a  
formula in \cite{N1}, \cite{N2} for a representative metric $g_F$
of the conformal class 
such that the components of $g_F$ and $g_F^{-1}$ are polynomials in $F$,
the derivatives of $F$ of orders $\leq 4$, and $F_{qq}^{-1}$, with 
coefficients which are universal functions of the local coordinates.  The
metric $g_F$ is flat for $F=q^2$.  
In \cite{N2}, he considered the case $F=q^2+\sum_{k=0}^6a_kp^k + bz$ with
$a_k$, $b\in \R$, 
and gave an explicit formula for the ambient metric $\gt_F$ in normal form
relative to $g_F$.  For these $g_F$, $g_\rho$ in \eqref{normalform} is a
polynomial in $\rho$ of degree $\leq 2$.  In \cite{LN}, Leistner-Nurowski
showed that the holonomy of $\gt_F$ is contained in $G_2$ for all values of
the $a_k$'s and $b$, and is equal to $G_2$ if one of $a_3$, $a_4$, $a_5$ 
or $a_6$ is nonzero.  That the holonomy is contained in $G_2$ is
proved by exhibiting explicitly a suitable parallel 3-form (or equivalently
a nonisotropic parallel spinor) for $\gt_F$.  

In \cite{HS}, Hammerl-Sagerschnig characterized those conformal structures
of signature $(2,3)$ which arise from a generic distribution $\cD$ by the
existence of a parallel tractor 3-form $\chi$ compatible with 
the tractor metric $h$ in the sense that
$$
(U\into \chi)\wedge (V\into \chi)\wedge \chi = \lambda\, h(U,V)\, dv,\qquad  
U,V\in \cT_x,
$$
where $dv$ denotes the tractor volume form and $\lambda$ is some positive
constant.  As explained in \cite{HS}, the existence of such a  
$\chi$, which is all that we need here,  
follows easily from the realization of $\Lambda^3\cT^*$ and its connection as
associated to the $(\fg_2,P)$ Cartan connection.  Namely, since the 3-form
$\varphi\in\Lambda^3\R^7{}^*$ defining $G_2$ is fixed by $P$, the constant
function $\varphi$ on $\cB$ is $P$-equivariant, so determines a
section $\chi$ of the associated bundle $\Lambda^3\cT^*$.  Since 
$\varphi$ is fixed by all of $G_2$ and the associated covariant derivative
is given in terms of the action of $\fg_2$ and differentiation by vector
fields on $\cB$, it follows that $\chi$ is parallel.  The compatibility
condition follows from the fact that $h$ has a similar realization 
as associated to the quadratic form determined by $\varphi$.    

Let $\cD$ be a generic 2-plane field on $M$, with $\cD$ and $M$
real-analytic.  In the following, $\gt$ will denote a  
real-analytic ambient metric for Nurowski's conformal structure associated
to $\cD$, with domain 
a sufficiently small dilation-invariant neighborhood $\cGt$ of $\cG$
diffeomorphic to $\R_+\times  M\times \R$.

\bigskip
\noindent
{\it Proof of Theorem~\ref{holonomycontained}.}
By the result of Hammerl-Sagerschnig, there is a nonzero
parallel section $\chi\in\Gamma(\Lambda^3\cT^*)$ compatible with the
tractor metric.  By Theorem~\ref{main}, $\chi$ has an extension as a
real-analytic 3-form 
$\ct$ on some dilation-invariant neighborhood $\cGt$ of $\cG$ which is
parallel with respect to $\gt$.  $\ct$ is compatible with $\gt$ by analytic
continuation, so it follows that $\Hol(\cGt,\gt)\subset G_2$.
\stopthm 

Next we discuss Cartan's basic curvature invariant of generic 2-plane
fields.  In Cartan's derivation this invariant arises directly from 
the structure equations.  We describe how it can be realized as a piece of
the Weyl tensor of the conformal structure.

Let $\cD$ be a generic 2-plane field on a (connected, oriented) 5-manifold
$M$ and let 
$\cD^1=\operatorname{span}(\cD,[\cD,\cD])$ be the derived 3-plane  
distribution.  We first define an isomorphism $\tau:\cD\rightarrow
TM/\cD^1$ invariantly up    
to scale.  Choose a metric $g$ in the conformal class.   
In the appendix it is observed that $\cD^1=\cD^\perp$, where ${}^\perp$
denotes orthogonal complement with respect to $g$.  Thus the map 
$\psi:TM/\cD^1\rightarrow \cD^*$ obtained 
by lowering an index with respect to $g$ and restricting to $\cD$ is 
well-defined and an isomorphism.  Now 
$g$ defines a negative definite metric on the line bundle $\cD^1/\cD$, and
this line bundle is trivial since $\cD$ is orientable.  
Let $\alpha$ be a section of $(\cD^1/\cD)^*$ with $g$-length-squared equal to
$-\frac34$; $\alpha$ is uniquely determined up to multiplication by $-1$.  
The Lie bracket induces a pointwise map $\cD\times \cD\rightarrow \cD^1/\cD$  
which we can regard as a scalar-valued nondegenerate skew form 
$\cD\times \cD\rightarrow \R$ by $(X,Y)\rightarrow \alpha([X,Y])$.  
This form induces an identification 
$\mu:\cD\rightarrow \cD^*$ defined by $\langle\mu(X),Y\rangle = \alpha([X,Y])$,
where $\langle\cdot,\cdot\rangle$ denotes the duality pairing.  Then 
$\tau=\psi^{-1}\circ \mu:\cD\rightarrow TM/\cD^1$ is our desired isomorphism.
Clearly $\langle\mu(X),X\rangle=0$ for $X\in \cD$.  Since $\psi$ turns the
$g$ pairing into the duality pairing, it follows that $\tau(X)\in
X^\perp/\cD^1$.  Since $X^\perp/\cD^1$ is 1-dimensional, $\tau(X)$ actually
spans $X^\perp/\cD^1$ if $X\neq 0$.
Under rescaling $\gh=\Omega^2 g$, we have $\widehat{\psi}=\Omega^2 \psi$ 
and $\widehat{\mu}=\Omega \mu$ so that $\widehat{\tau}=\Omega^{-1} \tau$.
Changing the sign of $\alpha$ multiplies $\mu$ and $\tau$ by $-1$.  

Let $W$ denote the Weyl tensor of our chosen metric $g$, viewed as a
covariant 4-tensor.  Proposition~\ref{Weylvanish} shows that 
$W(\cdot,\cdot,Y,Z)=0$ if $Y$, $Z\in \cD^1$.  Therefore we can
define a section $A$ of $\otimes^4\cD^*$ by 
\begin{equation}\label{Adef}
A(X_1,X_2,X_3,X_4)=W(\tau(X_1),X_2,\tau(X_3),X_4),\qquad X_i\in \cD_y
\end{equation}
since the right-hand side is independent of the $\cD^1$ ambiguity in $\tau$.  
Since  
$\widehat{W}=\Omega^2W$ and $\widehat{\tau}=\Omega^{-1}\tau$ and the right
hand side is invariant under $\tau\rightarrow -\tau$, it follows 
that $A$ is an absolute invariant of the 2-plane distribution $\cD$.   
Proposition~\ref{Asym} shows that $A\in \Gamma(S^4\cD^*)$ is symmetric.

The proof of Theorem~\ref{holonomycriterion} uses the following result
which follows from the arguments of Leistner-Nurowski.   

\begin{proposition}\label{LN}
Let $\cD$ be a real-analytic generic 2-plane field on a connected,
simply connected, real-analytic 5-manifold $M$.  
Suppose $\Hol(\gt)$ is strictly contained in $G_2$.  Then at 
least one of the following three conditions holds:  
\begin{enumerate}
\item
$\gt$ is locally symmetric.
\item
There is an open dense set $\cU\subset M$    
such that every point of $\cU$ has a neighborhood on which 
there is an Einstein metric in the conformal class $[g]$.  
\item
Every point of $M$ has a neighborhood on which there is a metric $g$
in the conformal class and a null line 
bundle $L\subset TM$ such that $L$ is parallel for $g$ and 
$Z\into \Ric_g=0$ for all $Z\in L^\perp$.  In this case, 
$W(U,K,K,Z)=0$ for all $K\in L$, $U\in TM$, $Z\in L^\perp$.  
\end{enumerate}
\end{proposition}

\noindent
We give a brief outline of the proof.  See \cite{LN} for 
details.  Berger's list contains all 
irreducibly acting holonomy groups of simply connected
non-locally-symmetric 
pseudo-Riemannian manifolds.  The only groups on the list in signature
$(3,4)$ are $G_2$ and $SO(3,4)$.  So if the holonomy is strictly contained
in $G_2$ and 
$\gt$ is not locally symmetric, its holonomy must act reducibly.  If
$V\subset \R^7$ is a nontrivial invariant subspace for the holonomy group,
its orthogonal $V^\perp$ is also invariant.  If $V\cap V^\perp=\{0\}$, 
the deRham decomposition theorem gives a local splitting of 
$\gt$ as a product metric.  This leads to condition (2).  If $V$ and
$V^\perp$ intersect nontrivially,  
their intersection determines a parallel totally null distribution of rank
at most 3 on the ambient space.  The cases where $\dim{(V\cap V^\perp)}=$ 1 
or 3 lead to 
condition (2).  The case where $\dim{(V\cap V^\perp)}=2$ leads 
to condition (3).  (These deductions to conditions (2) and (3) are    
substantial theorems in themselves.) 

As regards condition (3) of Proposition~\ref{LN}, we have the following
lemma. 
\begin{lemma}\label{WA}
Let $\cD$ be a generic 2-plane field on a 5-manifold $M$ and let $W$ 
be the Weyl 
tensor at $y\in M$ of a representative of Nurowski's associated conformal
structure.  There exists a nonzero null vector $K\in T_yM$ such that 
$W(U,K,K,Z)=0$ for all $U\in T_yM$, $Z\in K^\perp$ if and only if  
$A_y$ is 3-degenerate.
\end{lemma}
\begin{proof}
Suppose first that $A_y$ is 3-degenerate.  So
there exists $0\neq X\in \cD$ such that $A(Y,X,X,X)=0$ for all $Y\in \cD$.  
({From} now on we suppress writing ${}_y$.)   Just take $K=X$.
Certainly 
$X$ is nonzero and null.  In order to check $W(U,X,X,Z)=0$, by 
Proposition~\ref{Weylvanish} only the
equivalence classes of $U$, $Z \mod \cD^1$ are relevant.  We can write
$U+ \cD^1=\tau(Y)$ for some $Y\in \cD$, and modulo a multiplicative constant
can write $Z+ \cD^1=\tau(X)$.  Then
$W(U,X,X,Z)=W(\tau(Y),X,X,\tau(X))=-A(Y,X,X,X)=0$.      

For the converse, suppose that $K$ is nonzero and null and 
$W(U,K,K,Z)=0$ for all $U\in TM$, $Z\in K^\perp$.  First suppose $K\in \cD$.
In this case we take $X=K$, $Z+ \cD^1=\tau(X)$, $U+ \cD^1=\tau(Y)$ as 
above to deduce that $A(Y,X,X,X)=0$ for all $Y\in \cD$.  There are no 
null vectors in $\cD^1\setminus \cD$, so the only other possibility is $K\in
TM\setminus \cD^1$.  
In this case we can write $K+ \cD^1=\tau(X)$ for some nonzero $X\in \cD$.    
Take $Z=X$ and $U$ to be an arbitrary vector $Y\in \cD$ to deduce that 
$A(Y,X,X,X)=0$ for all $Y\in \cD$.  
\end{proof}

\noindent
{\it Proof of Theorem~\ref{holonomycriterion}.} 
Theorem~\ref{holonomycontained} shows that $\Hol(\gt)\subset G_2$.  So  
it suffices to show that the restriction of $\gt$ to some connected  
open subset of its domain has holonomy equal to $G_2$.  
We can choose a connected, simply connected subset of $M$
containing $x$ and $y$.  By replacing $M$ by this subset, we may as  
well assume that $M$ is simply connected.  Thus we can apply
Proposition~\ref{LN}.   
We show that injectivity of $L_x$ and 3-nondegeneracy of $A_y$ is
incompatible with each of conditions (1)-(3) of  
Proposition~\ref{LN}.  Note that $L$ is injective on an open set about $x$
since injectivity is invariant under perturbation.  

At a point where $A$ is 3-nondegenerate, it is in particular nonzero, so
the Weyl tensor $W$ of a representative metric $g$ is nonzero.  This is
enough to 
conclude that $\nt\Rt\neq 0$ so that $\gt$ is not locally symmetric.   
In fact, the second equation of \eqref{curv} with $r=1$ implies
$\nt_T\Rt=-2\Rt$, and one has $\Rt_{ijkl}=t^2W_{ijkl}$.  So if $W\neq 0$,  
then $\Rt\neq 0$, so $\nt\Rt\neq 0$.  

An Einstein metric has vanishing Cotton tensor.  By the transformation law
of the Cotton tensor, it follows that if $\gh=e^{2\omega}g$ is 
Einstein, then $W_{ijkl}\omega^i+C_{jkl}=0$.  Hence if (2) holds, then
$L$ is not injective on $\cU$.  Since $\cU$ is
dense, this is incompatible with injectivity of $L$ on an open set.     

If (3) holds, then Lemma~\ref{WA} shows that $A$ is 3-degenerate at all
points of $M$.  This violates the assumption that $A_y$ is
3-nondegenerate. 
\stopthm

\bigskip
\noindent
{\it Proof of Proposition~\ref{Fhol}.}
Represent $L$ at the origin as a matrix in some basis.  
Its rank 
is less than 6 if and only if the determinant of each of its $6\times 6$  
submatrices vanishes.  This clearly defines an algebraic subvariety in the
space of jets; we have only to show that it is proper.  
In the appendix of \cite{LN}, Leistner-Nurowski give explicit formulae for
the Weyl and Cotton tensors 
for the 8-parameter family $F=q^2+\sum_{k=0}^6a_kp^k + bz$.  It is
straightforward to check from their formulae that if $a_3\neq 0$ and
$a_4\neq 0$, then $L$ at the origin is injective.  (For instance, suppose
that the right 
hand side of our \eqref{Lform} vanishes for some $(v,\lambda)$.  Taking 
successively $jkl=415, 115, 413, 113, 414, 114$ shows the vanishing of, 
resp., $v^1$, $v^4$, $\lambda$, $v^3$, $v^2$, $v^5$.)  So
the subvariety where $\operatorname{rank}(L)<6$ is proper.  

Now $A$ at the origin is a symmetric 4-form on 
$\cD_0=\operatorname{span}\{\pa_q, \pa_x\}$.  If we represent $X\in \cD_0$ in 
terms of the coordinates $(u, v)$ 
dual to this basis, then the homogeneous polynomial defined by
$A$ takes the form
$$
A(X,X,X,X)=A_0u^4+4A_1u^3v+6A_2u^2v^2+4A_3uv^3+A_4v^4   
$$
for $A_0,\ldots, A_4\in \R$ (cf. \eqref{Aform}).  The conditions  
$A(\pa_q,X,X,X)=0$, $A(\pa_x,X,X,X)=0$ are 
\begin{gather}
\begin{gathered}
A_0u^3+3A_1u^2v+3A_2uv^2+A_3v^3=0\\
A_1u^3+3A_2u^2v+3A_3uv^2+A_4v^3=0.
\end{gathered}
\end{gather}
So $A$ is 3-degenerate if and only if this pair of equations has a common 
solution $(u,v)\in \R^2\setminus (0,0)$.  The set of $A_0,\ldots, A_4$ such
that this holds is contained in the set where the two equations have a
common solution in $\C^2\setminus (0,0)$, which is characterized by 
the vanishing of the resultant:
\begin{equation}\label{resultant}
\left|
\begin{matrix}
A_0&3A_1&3A_2&A_3&0&0\\
0&A_0&3A_1&3A_2&A_3&0\\
0&0&A_0&3A_1&3A_2&A_3\\
A_1&3A_2&3A_3&A_4&0&0\\
0&A_1&3A_2&3A_3&A_4&0\\
0&0&A_1&3A_2&3A_3&A_4
\end{matrix}
\right|=0.
\end{equation}
This equation defines an algebraic subvariety in the space of $7$-jets of
$F$ at the 
origin, and we must show it is proper.  We cannot do this by considering
Leistner-Nurowski's family of examples, because inspection of the formulae
in the appendix of \cite{LN} shows that the 
$A$ which arise from their family are everywhere 2-degenerate, i.e. for
any $F$ in their 8-parameter family, at each 
point there is $0\neq X\in \cD$ such that $A(Y_1,Y_2,X,X)=0$ for all
$Y_1$, $Y_2\in \cD$.  Instead we argue as follows. 

It is easily seen that \eqref{resultant} defines a proper subvariety in
$S^4(\cD_0^*)$ represented as the space of $A$'s.  To see  
that it defines a proper subvariety in the space of $7$-jets of $F$, it 
certainly suffices to show that the map from jets of $F$ at the origin to
$S^4(\cD_0^*)$ is surjective.  Take $F=q^2 + f$, where  
$f=f(x,y,z,p,q)=O(|(x,y,z,p,q)|^6)$.  For such $f$, we claim that 
$A$ can be identified with a nonzero constant multiple of 
$\left(\nabla^4\pa_q^2f\right)(0)|_{\cD_0}$.  Here $\pa_q^2f$ vanishes to
order 
4 at the origin, $\left(\nabla^4\pa_q^2f\right)(0)$ denotes the
symmetric 4-tensor on $T_0M$ defined by the order 4 Taylor
polynomial of $\pa_q^2f$ at the origin, and $|_{\cD_0}$ its restriction to 
$\cD_0$.  (Recall that if a smooth function $\varphi$ 
on a manifold $M$ vanishes to order $k$ at a point $y$, then
$\nabla^k\varphi(y)$ 
is an invariantly defined symmetric $k$-form on $T_yM$ depending only on
the smooth structure.)  Up to an overall nonzero constant multiple, the
$A_i$ above are given by 
\begin{equation}\label{Apartial}
A_0=\pa_q^6f(0),\; A_1=\pa_q^5\pa_xf(0),\;
A_2=\pa_q^4\pa_x^2f(0),\;  A_3=\pa_q^3\pa_x^3f(0),\;  
A_4=\pa_q^2\pa_x^4f(0).
\end{equation}
This follows by direct calculation.  Equation (1.3) in \cite{N2} gives a
formula for a representative of the conformal structure in terms of $F$.
All terms in the formula involve at most 4 derivatives of
$F$.  For $F=q^2+f$ as above, the only terms which can contribute to the
value of the curvature tensor at the origin must involve 4 derivatives of
$F$.  Inspecting term by term shows that the curvature at the origin of
$g_F$ is the same as the curvature at the origin of the metric $g_{q^2}+h$,
where 
$$
h=-24 \pa_x^2\pa_q^2f\,dy^2
+ 24 \pa_x\pa_q^3f\,dydz-6\pa_q^4f\,dz^2.
$$
Since $g_{q^2}$ is flat, the curvature tensor at the origin of $g_{q^2}+h$
is given by 
$$
R_{ijkl}=\tfrac12 \left( h_{il,jk}-h_{jl,ik}-h_{ik,jl}+h_{jk,il}\right),
$$
where the indices correspond to components and derivatives with respect to
the frame $\{\pa_y,\pa_z,\pa_p,\pa_q,\pa_x\}$.  Using the fact that 
$g_{q^2}(0)=480dydq+240dzdx-320dp^2$, it is straightforward but 
tedious to identify the isomorphism $\tau$ appearing in \eqref{Adef}, to
calculate the relevant components of the Weyl tensor at the origin, and
then to verify \eqref{Apartial}.  
\stopthm

\section{Appendix}\label{appendix}

In this appendix we collect facts about Cartan's connection \cite{C}
associated to 
generic 2-plane fields in the form given by Nurowski \cite{N1} (modulo some
relabeling).  Other
discussions may be found in the literature.

Define $\varphi\in \Lambda^3\R^7{}^*$ by
$$
\varphi = 6dx^{012} +\sqrt{3}\left(
dx^{234}-dx^{135}+dx^{036}\right)+dx^{456} 
$$
where the coordinates are labeled $(x^0,x^1,\cdots,x^6)$ and
$dx^{ijk}=dx^i\wedge dx^j\wedge dx^k$.  
Define $G_2=\{A\in GL(7,\R): A^*\varphi = \varphi\}$.     
Then $G_2\subset SO(\widetilde{h})$, where 
$$
\widetilde{h}_{IJ}=
\begin{pmatrix}
0&0&1\\
0&h_{ij}&0\\
1&0&0
\end{pmatrix}
$$
and
$$
h_{ij}=
\begin{pmatrix}
0&0&0&0&-1\\
0&0&0&1&0\\
0&0&-1&0&0\\
0&1&0&0&0\\
-1&0&0&0&0\\
\end{pmatrix}.
$$
The Lie algebra $\fg_2\subset\fs\fo(\widetilde{h})$ is the set of 
matrices of the form 
$$
\begin{pmatrix}
-(a_1+a_4)&a_8&a_9&-\frac{1}{\sqrt{3}}a_7&\frac{1}{2\sqrt{3}}a_5&
\frac{1}{2\sqrt{3}}a_6&0\\
b^1&a_1&a_2&\frac{1}{\sqrt{3}}b^4&-\frac{1}{2\sqrt{3}}b^3&
0&\frac{1}{2\sqrt{3}}a_6\\
b^2&a_3&a_4&\frac{1}{\sqrt{3}}b^5&0&-\frac{1}{2\sqrt{3}}b^3&
-\frac{1}{2\sqrt{3}}a_5\\
b^3&a_5&a_6&0&\frac{1}{\sqrt{3}}b^5&-\frac{1}{\sqrt{3}}b^4&
-\frac{1}{\sqrt{3}}a_7\\
b^4&a_7&0&a_6&-a_4&a_2&-a_9\\
b^5&0&a_7&-a_5&a_3&-a_1&a_8\\
0&b^5&-b^4&b^3&-b^2&b^1&a_1+a_4
\end{pmatrix}.
$$
Define $P=\{A\in G_2:Ae_0=\lambda e_0, \lambda>0\}$.  Its Lie algebra 
$\fp\subset \fg_2$ is the subset given by $b^1=\cdots =b^5=0$. 
The quadratic form $h_{ij}b^ib^j=-2b^1b^5+2b^2b^4-(b^3)^2$ on $\fg_2/\fp$
is preserved up to scale by the adjoint action of $P$.

Let $\cD\subset TM$ be a generic 2-plane field on a connected, oriented
5-manifold $M$.  There is a principal bundle $\cB\rightarrow M$
with structure group $P$ and Cartan connection $\omega:T\cB\rightarrow
\fg_2$ canonically associated to $\cD$ up to equivalence.  The Cartan
connection can be written 
$$
\omega=
\begin{pmatrix}
-(\fe_1+\fe_4)&\fe_8&\fe_9&-\frac{1}{\sqrt{3}}\fe_7&\frac{1}{2\sqrt{3}}\fe_5&
\frac{1}{2\sqrt{3}}\fe_6&0\\
\theta^1&\fe_1&\fe_2&\frac{1}{\sqrt{3}}\theta^4&-\frac{1}{2\sqrt{3}}\theta^3&
0&\frac{1}{2\sqrt{3}}\fe_6\\
\theta^2&\fe_3&\fe_4&\frac{1}{\sqrt{3}}\theta^5&0&-\frac{1}{2\sqrt{3}}\theta^3&
-\frac{1}{2\sqrt{3}}\fe_5\\
\theta^3&\fe_5&\fe_6&0&\frac{1}{\sqrt{3}}\theta^5&-\frac{1}{\sqrt{3}}\theta^4&
-\frac{1}{\sqrt{3}}\fe_7\\
\theta^4&\fe_7&0&\fe_6&-\fe_4&\fe_2&-\fe_9\\
\theta^5&0&\fe_7&-\fe_5&\fe_3&-\fe_1&\fe_8\\
0&\theta^5&-\theta^4&\theta^3&-\theta^2&\theta^1&\fe_1+\fe_4
\end{pmatrix}
$$
where the $\theta^i$ and $\fe_j$ are scalar 1-forms on $\cB$.  If $\sigma$
is a local section of $\cB\rightarrow M$, set $\tb{}^i=\sigma^*\theta^i$.
Then $\{\tb{}^1,\tb{}^2,\tb{}^3,\tb{}^4,\tb{}^5\}$ is a frame for 
$T^*M$ for which $\cD=\ker\{\tb{}^1,\tb{}^2,\tb{}^3\}$ and the derived
distribution $\cD^1=\operatorname{span}(\cD,[\cD,\cD])$ is given by 
$\cD^1=\ker\{\tb{}^1,\tb{}^2\}$.  The metric 
$g=h_{ij}\tb{}^i\tb{}^j =-2\tb{}^1\tb{}^5+2\tb{}^2\tb{}^4-(\tb{}^3)^2$ is a
representative   
for Nurowski's conformal structure.  It is clear that $\cD^\perp=\cD^1$
with respect to such a metric.  

The curvature $\Omega=d\omega + \omega\wedge\omega$ has the form 
\begin{equation}\label{cartancurvature}
\Omega=
\begin{pmatrix}
0&\Phi_8&\Phi_9&\frac{1}{\sqrt{3}}\Phi_7&\frac{1}{2\sqrt{3}}\Phi_5&
\frac{1}{2\sqrt{3}}\Phi_6&0\\
0&\Phi_1&\Phi_2&0&0&
0&\frac{1}{2\sqrt{3}}\Phi_6\\
0&-\Phi_3&-\Phi_1&0&0&0&
-\frac{1}{2\sqrt{3}}\Phi_5\\
0&\Phi_5&\Phi_6&0&0&0&
\frac{1}{\sqrt{3}}\Phi_7\\
0&-\Phi_7&0&\Phi_6&\Phi_1&\Phi_2&-\Phi_9\\ 
0&0&-\Phi_7&-\Phi_5&-\Phi_3&-\Phi_1&\Phi_8\\
0&0&0&0&0&0&0
\end{pmatrix}
\end{equation}
where
\begin{equation}\label{Phiform}
\begin{split}
\Phi_1&= C_1\theta^1\wedge\theta^2
+B_1\theta^1\wedge\theta^3   
+B_2\theta^2\wedge\theta^3
+A_1\theta^1\wedge\theta^4\\
&\qquad\qquad\qquad\qquad\qquad\qquad\qquad\qquad
+A_2\theta^1\wedge\theta^5
+A_2\theta^2\wedge\theta^4
+A_3\theta^2\wedge\theta^5\\
\Phi_2&= C_2\theta^1\wedge\theta^2
+B_2\theta^1\wedge\theta^3   
+B_3\theta^2\wedge\theta^3
+A_2\theta^1\wedge\theta^4\\
&\qquad\qquad\qquad\qquad\qquad\qquad\qquad\qquad
+A_3\theta^1\wedge\theta^5
+A_3\theta^2\wedge\theta^4
+A_4\theta^2\wedge\theta^5\\
\Phi_3&= C_0\theta^1\wedge\theta^2
+B_0\theta^1\wedge\theta^3   
+B_1\theta^2\wedge\theta^3
+A_0\theta^1\wedge\theta^4\\
&\qquad\qquad\qquad\qquad\qquad\qquad\qquad\qquad
+A_1\theta^1\wedge\theta^5
+A_1\theta^2\wedge\theta^4
+A_2\theta^2\wedge\theta^5\\
\Phi_5&=D_0\theta^1\wedge\theta^2
+2C_0\theta^1\wedge\theta^3   
+2C_1\theta^2\wedge\theta^3
+ B_0\theta^1\wedge\theta^4\\
&\qquad\qquad\qquad\qquad\qquad\qquad\qquad\qquad
+ B_1\theta^1\wedge\theta^5
+ B_1\theta^2\wedge\theta^4
+ B_2\theta^2\wedge\theta^5\\
\Phi_6&=D_1\theta^1\wedge\theta^2
+2C_1\theta^1\wedge\theta^3   
+2C_2\theta^2\wedge\theta^3
+B_1\theta^1\wedge\theta^4\\
&\qquad\qquad\qquad\qquad\qquad\qquad\qquad\qquad
+ B_2\theta^1\wedge\theta^5
+ B_2\theta^2\wedge\theta^4
+ B_3\theta^2\wedge\theta^5\\
\Phi_7&=E\theta^1\wedge\theta^2
+D_0\theta^1\wedge\theta^3   
+D_1\theta^2\wedge\theta^3
+ C_0\theta^1\wedge\theta^4\\
&\qquad\qquad\qquad\qquad\qquad\qquad\qquad\qquad
+ C_1\theta^1\wedge\theta^5
+ C_1\theta^2\wedge\theta^4
+ C_2\theta^2\wedge\theta^5.
\end{split}
\end{equation}
The coefficients $A_0$, $A_1$, $A_2$, $A_3$, $A_4$, 
$B_0$, $B_1$, $B_2$, $B_3$, 
$C_0$, $C_1$, $C_2$, $D_0$, $D_1$, $E$ are Cartan's curvature quantities.  
There are further formulae for $\Phi_8$, $\Phi_9$ which we will not need
here; see Nurowski.  

The Cartan geometry $(\cB,\omega)$ may be regarded as the reduction of a  
Cartan geometry $(\cB_c,\omega_c)$ of type
$(\fs\fo(\widetilde{h}),P_c)$, where 
$P_c=\{A\in SO(\widetilde{h}):Ae_0=\lambda e_0, \lambda>0\}$.  Observe
that \eqref{cartancurvature} may be written
\begin{equation}\label{Oconformal}
\Omega=
\begin{pmatrix}
0&\Omega_j&0\\
0&\Omega^i{}_j&-\Omega^i\\
0&0&0
\end{pmatrix}
\end{equation}
where
$$
\Omega^i{}_j=
\begin{pmatrix}
\Phi_1&\Phi_2&0&0&0\\
-\Phi_3&-\Phi_1&0&0&0\\
\Phi_5&\Phi_6&0&0&0\\
-\Phi_7&0&\Phi_6&\Phi_1&\Phi_2\\
0&-\Phi_7&-\Phi_5&-\Phi_3&-\Phi_1
\end{pmatrix},
$$
\vspace{.1in}
$$
\Omega_j=
\begin{pmatrix}
\Phi_8&\Phi_9&\frac{1}{\sqrt{3}}\Phi_7&\frac{1}{2\sqrt{3}}\Phi_5&
\frac{1}{2\sqrt{3}}\Phi_6
\end{pmatrix},
$$
and $\Omega^i=h^{ij}\Omega_j$.  One checks directly from  
\eqref{Phiform} that the coefficients $\Omega^i{}_{jkl}$ defined by 
$$
\Omega^i{}_j=\tfrac12\Omega^i{}_{jkl}\theta^k\wedge\theta^l\qquad\qquad
\Omega^i{}_{jkl}=-\Omega^i{}_{jlk} 
$$ 
satisfy $\Omega^i{}_{jil}=0$.  It follows (see, e.g., \cite{Ko}) that  
$\omega_c$ is the normal Cartan connection for Nurowski's conformal 
structure.  Lowering an index gives  
$\Omega_{ij}=\frac12 \Omega_{ijkl}\theta^k\wedge\theta^l$:   
\begin{equation}\label{Wform}
\Omega_{ij}=
\begin{pmatrix}
0&\Phi_7&\Phi_5&\Phi_3&\Phi_1\\ 
-\Phi_7&0&\Phi_6&\Phi_1&\Phi_2\\
-\Phi_5&-\Phi_6&0&0&0\\
-\Phi_3&-\Phi_1&0&0&0\\
-\Phi_1&-\Phi_2&0&0&0
\end{pmatrix}.
\end{equation}
(This corrects the corresponding formula given in the appendix of
\cite{N1}.)
{From} \eqref{Phiform} and \eqref{Wform} one reads off the following:
\begin{equation}\label{As}
\begin{split}
A_0&=\Omega_{1414}\\
A_1&=\Omega_{1415}=\Omega_{1424}=\Omega_{1514}=\Omega_{2414}\\
A_2&=\Omega_{1425}=\Omega_{1515}=\Omega_{1524}=\Omega_{2415}=\Omega_{2424}=\Omega_{2514}\\
A_3&=\Omega_{1525}=\Omega_{2425}=\Omega_{2515}=\Omega_{2524}\\
A_4&=\Omega_{2525}.
\end{split}
\end{equation}

If $\sigma$ is a section of $\cB\rightarrow M$ as above, then comparing
\eqref{Oconformal} with the form of the curvature of the normal conformal 
Cartan connection shows that  
$W_{ijkl}=\Omega_{ijkl}\circ\sigma$ are the components of the Weyl tensor
of the metric $g=h_{ij}\tb{}^i\tb{}^j$  
in the frame $\{\tb{}^1,\tb{}^2,\tb{}^3,\tb{}^4,\tb{}^5\}$.  Similarly,
setting $\Omega_j=\frac12 \Omega_{jkl}\theta^k\wedge\theta^l$ with
$\Omega_{jkl}=-\Omega_{jlk}$, the components of the Cotton tensor are given
by $C_{jkl}=\Omega_{jkl}\circ\sigma$.  
Thus this expresses the Weyl and Cotton curvature in terms of Cartan's  
scalar invariants.

\begin{proposition}\label{Weylvanish}
$W(\cdot,\cdot,Y,Z)=0$ if $Y$, $Z\in \cD^1$.   
\end{proposition}
\begin{proof}
This is clear since $\Omega_{ijkl}=0$ if $i$, $j\geq 3$ from
\eqref{Wform}.  (Or observe from \eqref{Phiform} that
$\Omega_{ij}=0\mod \theta^1,\theta^2$ for all $i$, $j$.)
\end{proof}

\begin{proposition}\label{Asym}
The tensor $A$ defined in \eqref{Adef} is symmetric,
i.e. $A\in \Gamma(S^4\cD^*)$.    
\end{proposition}
\begin{proof}
Consider the maps $\psi$, $\mu$, $\tau$ defined in the paragraph before
\eqref{Adef}.  
Let again $\{\tb{}^1,\tb{}^2,\tb{}^3,\tb{}^4,\tb{}^5\}$ be a coframe
on $M$ obtained by pulling back by a local section of the principal bundle 
and let $\{U_1,U_2,U_3,U_4,U_5\}$ be the dual frame.  Then 
$\cD=\operatorname{span}\{U_4, U_5\}$ and 
$\cD^1=\operatorname{span}\{U_3, U_4, U_5\}$.  
Lowering an index shows that  
$$
\psi(U_1+ \cD^1)=-\tb{}^5|_\cD,\qquad \psi(U_2 +\cD^1)=\tb{}^4|_\cD.
$$ 
Taking $\alpha
=\frac{\sqrt{3}}{2}\;\tb{}^3$, consideration of the ${}^0{}_3$ component of
$\Omega=d\omega+\omega\wedge\omega$ shows that 
$$
\mu(U_4)=-\tb{}^5|_\cD,\qquad 
\mu(U_5)=\tb{}^4|_\cD. 
$$ 
Thus  
$$
\tau(U_4)=U_1 + \cD^1, \qquad 
\tau(U_5)=U_2 + \cD^1.
$$  
Hence \eqref{As} shows that $A$ defined by \eqref{Adef} is symmetric.   
\end{proof}

\noindent
Note that \eqref{As} in fact shows that  
\begin{equation}\label{Aform}
A=A_0u^4+4A_1u^3v+6A_2u^2v^2+4A_3uv^3+A_4v^4,
\end{equation}
where $u=\tb^4|_{\cD}$, $v=\tb^5|_{\cD}$, and we abuse notation by denoting
also by $A_0$, $A_1$, $A_2$, $A_3$, $A_4$ their pullbacks under the local
section $\sigma$ of $\cB$. 

The reductive part of $P$ is $\R_+\cdot SL(2,\R)$, where $SL(2,\R)$ is 
viewed as a subgroup of $SO(h)$ via 
$$
SL(2,\R)\ni S \mapsto
\begin{pmatrix}
S&0&0\\
0&1&0\\
0&0&S
\end{pmatrix}\in SO(h).
$$
The space of Weyl tensors is a 35-dimensional irreducible representation
for $SO(h)$.  Upon restriction to $SL(2,\R)$ it decomposes as
$S^4\oplus 2S^3 \oplus 3S^2 \oplus 4S^1 \oplus 5\R$, where $S^k$
denotes the $k^{\text{th}}$ symmetric power of the standard representation
and $\R=S^0$ the trivial representation.  The space of Weyl tensors which
arise from a conformal structure associated to a generic 2-plane
distribution is the 15-dimensional subspace given by \eqref{Phiform}, 
\eqref{Wform}.  This decomposes under $SL(2,\R)$ 
as $S^4\oplus S^3 \oplus S^2 \oplus S^1 \oplus \R$ corresponding to the 
division of Cartan's scalar invariants into $A$'s, $B$'s, $C$'s, $D$'s, and 
$E$.


\begin{thebibliography}{BEGM}


\bibitem[BEG]{BEG} T. N. Bailey, M. G. Eastwood and A. R. Gover, {\it 
Thomas's structure bundle for conformal, projective and related
structures}, Rocky Mountain J. Math. {\bf 24} (1994), 1191--1217. 

\bibitem[B]{B} O. Biquard, {\it Continuation unique \`a partir de l'infini
conforme pour les m\'etriques d'Einstein}, Math. Res. Lett. 
{\bf 15} (2008), 1091--1099. {\tt arXiv:0708.4346}.  

\bibitem[BG]{BG} T. Branson and A. R. Gover, {\it Conformally invariant
operators, differential forms, cohomology and a generalisation of
Q-curvature}, Comm. P. D. E. {\bf 30} (2005), 1611--1669,
{\tt arXiv:math/0309085}. 

\bibitem[Br1]{Br1} R. Bryant, {\it Metrics with exceptional holonomy}, 
Ann. Math. {\bf 126} (1987), 525--576.

\bibitem[Br2]{Br2} R. Bryant, {\it Conformal geometry and 3-plane fields on
  6-manifolds},  Developments of Cartan Geometry and Related Mathematical
  Problems, RIMS Symposium Proceedings (Kyoto University) {\bf 1502},
  2006, 1--15, {\tt arXiv:math/0511110}.

\bibitem[BH]{BH} R. L. Bryant and L. Hsu, {\it Rigidity of integral curves
  of rank 2 distributions}, Invent. Math. {\bf 114} (1993), 435--461.

\bibitem[\v{C}G1]{CG1} A. \v{C}ap and A. R. Gover, {\it Standard
  tractors and the conformal ambient metric construction}, Ann. Global
  Anal. Geom. {\bf 24} (2003), 231--259, {\tt arXiv:math/0207016}.   

\bibitem[\v{C}G2]{CG2} A. \v{C}ap and A. R. Gover, {\it A holonomy 
  characterisation of Fefferman spaces}, Ann. Global Anal. Geom. {\bf 38}
  (2010), 399--412, {\tt arXiv:math/0611939}.

\bibitem[C]{C} \'E. Cartan, {\it Les syst\`emes de Pfaff, \`a cinq variables et
les \'equations aux d\'eriv\'ees partielles du second ordre}, 
Ann. Sci. \'Ecole Norm. Sup. {\bf 27} (1910), 109--192.

\bibitem[F]{F} C. Fefferman, {\it Monge-Amp\`ere equations, the Bergman
kernel, and geometry of pseudoconvex domains}, Ann. Math. {\bf 103} (1976),
395--416; Correction:  Ann. Math. {\bf 104} (1976), 393--394.  

\bibitem[FG1]{FG1} C. Fefferman and C. R. Graham, {\it
Conformal invariants,} in {\it The Mathematical Heritage of \'Elie Cartan 
(Lyon, 1984)},  Ast\'erisque, 1985, Numero Hors Serie, 95--116.  

\bibitem[FG2]{FG2} C. Fefferman and C. R. Graham, {\it
The Ambient Metric}, to appear, Princeton University Press, 
{\tt arXiv:math/0710.0919}.

\bibitem[Go1]{Go1} A. R. Gover, {\it Almost conformally Einstein manifolds
  and obstructions}, Differential geometry and its applications,
  Proceedings of the 9th International Conference (DGA 2004), 
  Matfyzpress, Prague, 2005, 247--260, {\tt arXiv:math/0412393}.

\bibitem[Go2]{Go2} A. R. Gover, {\it Laplacian operators and $Q$-curvature
on conformally Einstein manifolds},  Math. Ann. {\bf 336} (2006), 311--334, 
{\tt arXiv:math/0506037}.   

\bibitem[Go3]{Go3} A. R. Gover, {\it Almost Einstein and
Poincar\'e-Einstein manifolds in Riemannian signature},
J. Geom. Phys. {\bf 60} (2010), 182--204, {\tt arXiv:0803.3510}.      

\bibitem[Gr1]{Gr1} C.R. Graham, {\it Scalar boundary invariants and the
Bergman kernel}, Complex Analysis II, Proceedings, Univ. of Maryland
1985-86, Springer Lecture Notes {\bf 1276}, 108--135.  

\bibitem[Gr2]{Gr2} C.R. Graham, {\it Higher asymptotics of the complex
Monge-Amp\`ere equation}, Comp. Math. {\bf 64} (1987), 133--155. 

\bibitem[GH]{GH} C. R. Graham and K. Hirachi, {\it Inhomogeneous ambient
metrics}, Symmetries and Overdetermined Systems of Partial Differential
  Equations, 403--420, IMA Vol. Math. Appl. {\bf 144}, Springer, 2008, 
  {\tt arXiv:math/0611931}.    

\bibitem[GL]{GL} C. R. Graham and J. M. Lee, {\it Einstein metrics with 
prescribed conformal infinity on the ball},  
Adv. Math.  {\bf 87} (1991), 186--225.

\bibitem[H]{H} M. Hammerl, {\it  Invariant prolongation of BGG-operators in
  conformal geometry}, Arch. Math. (Brno) {\bf 44} (2008), 367--384, 
{\tt arXiv:0811.4122}.

\bibitem[HS]{HS} M. Hammerl and K. Sagerschnig, {\it Conformal structures 
associated to generic rank 2 distributions on 5-manifolds--characterization
and Killing-field decomposition}, SIGMA {\bf 5} (2009), 081, 29 pages, 
{\tt arXiv:0908.0483}.  


\bibitem[J]{J} A. Juhl, {\it Explicit formulas for GJMS-operators and
  $Q$-curvatures}, {\tt arXiv:1108.0273}.  

\bibitem[Ki]{Ki} S. Kichenassamy, {\it Fuchsian Reduction}, Birkh\"auser,
  2007. 

\bibitem[Ko]{Ko} S. Kobayashi, {\it Transformation Groups in Differential
  Geometry}, Springer, 1972.  

\bibitem[Lee]{Lee} J. Lee, {\it  Higher asymptotics of the complex
Monge-Amp\`ere equation and geometry of CR-manifolds}, MIT Ph.D. thesis, 
1982.  

\bibitem[Leis]{Leis} T. Leistner, {\it Conformal holonomy of C-spaces, 
Ricci-flat, and Lorentzian manifolds}, 
Diff. Geom. Appl. {\bf 24} (2006), 458--478, {\tt arXiv:math/0501239}. 

\bibitem[LN]{LN} T. Leistner and P. Nurowski, {\it Conformal structures
with $G_{2(2)}$-ambient metrics}, {\tt arXiv:0904.0186}.

\bibitem[Leit1]{Leit1} F. Leitner, {\it Normal conformal Killing forms}, 
{\tt arXiv:math/0406316}.  

\bibitem[Leit2]{Leit2} F. Leitner, {\it Conformal Killing forms with
  normalisation condition}, Rend. Circ. Mat. Palermo (2) Suppl. No. 75
  (2005), 279--292. 

\bibitem[Leit3]{Leit3} F. Leitner, {\it A remark on unitary conformal 
holonomy}, Symmetries and Overdetermined Systems of Partial Differential
  Equations, 445--460, IMA Vol. Math. Appl. {\bf 144}, Springer, 2008, 
  {\tt arXiv:math/0604393}.

\bibitem[N1]{N1} P. Nurowski, {\it Differential equations and conformal
structures}, J. Geom. Phys. {\bf 55} (2005), 19--49, {\tt
arXiv:math/0406400}. 

\bibitem[N2]{N2} P. Nurowski, {\it Conformal structures with explicit
ambient metrics and conformal $G_2$ holonomy}, Symmetries and
Overdetermined Systems of Partial Differential Equations, 515--526, IMA
Vol. Math. Appl. {\bf 144}, Springer, 2008, {\tt arXiv:math/0701891}.  

\bibitem[T]{T} T. Y. Thomas, {\it On conformal geometry},
  Proc. Natl. Acad. Sci. USA {\bf 12} (1926), 352--359. 

\bibitem[W]{W} T. Willse, {\it Parallel tractor extension and metrics of
  split $G_2$ holonomy}, University of Washington Ph.D. thesis, 2011. 


\end{thebibliography}
\end{document}